\documentclass[preprint,12pt]{elsarticle}

\usepackage{amssymb}
\usepackage{amsmath}
\usepackage[utf8]{inputenc}
\usepackage{amsfonts}
\usepackage{amsmath}
\usepackage{amsmath, amssymb}
\usepackage{amsthm}
\usepackage{mathtools}
\usepackage{enumerate}
\usepackage{comment}
\usepackage{xcolor}

\usepackage{chngcntr}
\usepackage{apptools}
\AtAppendix{\counterwithin{thm}{section}}

\theoremstyle{plain}
\newtheorem{thm}{Theorem}[section]
\newtheorem{prop}[thm]{Proposition}
\newtheorem{cor}[thm]{Corollary}
\newtheorem{lem}[thm]{Lemma}
\theoremstyle{definition}

\newtheorem*{nota}{Notation}

\theoremstyle{remark}

\newcommand{\p}{\partial}

\newcommand{\Bo}[1]{\dot{B}^{#1}_{2,1}}
\newcommand{\Bi}[1]{\dot{B}^{#1}_{2,\infty}}

\newcommand{\dive}{\mathrm{div}}
\newcommand{\delj}{\dot{\Delta}_{j}}

\newcommand{\deln}{\dot{\Delta}_{n}}
\newcommand{\deljj}{\dot{\Delta}_{j'}}

\newcommand{\delo}{\dot{\Delta}_{0}}

\newcommand{\ep}{\epsilon}
\newcommand{\vep}{v^*_\ep}
\newcommand{\bep}{b^*_\ep}
\newcommand{\vertiii}[1]{{\left\vert\kern-0.25ex\left\vert\kern-0.25ex\left\vert #1 
    \right\vert\kern-0.25ex\right\vert\kern-0.25ex\right\vert}}

\begin{document}

\begin{frontmatter}



\title{Low Mach number limit for the compressible Navier-Stokes equation with a stationary force} 


\author{Naoto Deguchi}
\ead{deguchi.naoto.7y@kyoto-u.ac.jp}

\affiliation{
  organization={Research Institute for Mathematical Sciences, Kyoto University},
  city={Kyoto},
  country={Japan}
}

\begin{abstract}
In this paper, we are concerned with the low Mach number limit for the compressible Navier-Stokes equation with a stationary force and ill-prepared initial data in the three-dimensional whole space. The convergence result of the stationary solutions toward the corresponding incompressible flow is obtained when the stationary force is small enough. Under the assumption that the initial perturbation around the stationary solution is small enough, the convergence result of the perturbation toward the corresponding perturbation around the stationary incompressible flow is obtained globally in time. The proof relies crucially on the Strichartz type estimate for the linearized semigroup around the motionless state which reflects not only its dispersive property but also dissipative properties of the linearized operator.

\medskip
\noindent\textbf{R\'esum\'e}\\
Dans cet article, nous \'etudions la limite de faible nombre de Mach pour les \'equations de Navier-Stokes compressibles avec une force stationnaire et des donn\'ees initiales mal pr\'epar\'ees dans l'espace tridimensionnel tout entier. Nous obtenons le r\'esultat de convergence des solutions stationnaires vers l'\'ecoulement incompressible correspondant lorsque la force stationnaire est suffisamment petite. Sous l'hypoth\`ese que la perturbation initiale autour de la solution stationnaire est suffisamment petite, nous \'etablissons globalement en temps la convergence de cette perturbation vers la perturbation correspondante autour de l'\'ecoulement incompressible stationnaire. La d\'emonstration repose de mani\`ere essentielle sur une estimation de type Strichartz pour le semi-groupe lin\'earis\'e autour de l'\'etat de repos, laquelle refl\`ete non seulement ses propri\'et\'es dispersives, mais aussi les propri\'et\'es dissipatives de l'op\'erateur lin\'earis\'e.

\end{abstract}


\end{frontmatter}




\section{Introduction}

We consider the Cauchy problem for the isentropic compressible Navier-Stokes equation:
\begin{equation} \label{eq}
    \left\{ \,
    \begin{aligned}
        &\partial_{t}\rho_\epsilon + \dive (\rho_\epsilon v_\epsilon) = 0, \\
        &\partial_{t}(\rho_\epsilon v_\epsilon) + \dive(\rho_\epsilon v_\epsilon \otimes v_\epsilon) = \mu \Delta v_\epsilon + (\mu+\mu')\nabla \dive\,v_\epsilon -\frac{\nabla P(\rho_\epsilon)}{\epsilon^2} + \rho_\epsilon F\\
        &(\rho_\ep ,v_\ep)|_{t=0}=(\rho_{\ep,0},v_{\ep,0}),\ \ \ \lim_{|x|\to \infty} (\rho_\ep , v_\ep)(t,x)=(\rho_\infty,0).
    \end{aligned}
    \right.
\end{equation}
Here $t\geq 0$, $x\in\mathbb{R}^3$, $\rho_\ep(t,x)$ is the fluid density, $v_\ep(t,x)=(v_{\ep,1},v_{\ep,2},v_{\ep,3})(t,x)$ is the fluid velocity and $\rho_\infty>0$ is a given constant. The positive number $\ep$ is called Mach number.  We assume that viscous coefficients $\mu$, $\mu'$ are constant and satisfying $\mu>0$, $2/3 \mu +\mu' \geq 0$ and the barotropic pressure $P(\cdot)$ satisfies $P \in C^{\infty}(\mathbb{R})$ and $P'(\rho_\infty)>0$. {  The function $F$ is a body force, which arises from the external effects acting on the fluid, such as gravitational and electromagnetic forces. (Cf. \cite{electro}, \cite{atmos}.) The system of equations in (\ref{eq}) is classified as a class of quasilinear hyperbolic-parabolic systems. Due to the parabolic aspect of (\ref{eq}), it could be regarded as a dissipative system. Consequently, when an external energy is exerted to the system, it drives the motion of the fluid and certain nontrivial flow patterns appear and persist; for example, such patterns appear in controlled laboratory experiments (e.g., \cite{ikehata1996formation}, \cite{weier2004control}, \cite{rashidi2016vortex} and \cite{zhao2023review}). When the external force $F$ is a stationary force, the generated pattern is described as a solution to the following stationary problem: }
\begin{equation} \label{stationary_eq}
    \left\{ \,
    \begin{aligned}
        &\dive (\rho^*_\epsilon v^*_\epsilon) = 0, \\
        &\dive(\rho^*_\epsilon v^*_\epsilon \otimes v^*_\epsilon) = \mu \Delta v^*_\epsilon + (\mu+\mu')\nabla \dive\,v^*_\epsilon -\frac{\nabla P(\rho^*_\epsilon)}{\epsilon^2} + \rho^*_\epsilon F(x),\\
        &\lim_{|x|\to \infty} (\rho_\ep^* , v_\ep^*)(x)=(\rho_\infty,0),
    \end{aligned}
    \right.
\end{equation}
where $x\in\mathbb{R}^3$, $\rho^*_\ep(x)$ is the fluid density and $v^*_\ep(x)=(v^*_{\ep,1},v^*_{\ep,2},v^*_{\ep,3})(x)$ is the fluid velocity.

The aim of this paper is to study the low Mach number limit of the solution of the Cauchy problem (\ref{eq}) around the stationary solution of (\ref{stationary_eq}). We show that if a stationary force $F(x)$ is small in the function space $\Bi{-3/2}\cap \dot{H}^4$, then there exists a unique stationary solution $(\rho^*_\ep, v^*_\ep)$ such that $\rho^*_\ep$ converges to $\rho_\infty$ and $v^*_\ep$ converges to the solution $u^*$ of the stationary problem of the incompressible Navier-Stokes equation:
\begin{equation} \label{istationary_eq}
    \left\{ \,
    \begin{aligned}
        &\rho_\infty \dive(u^* \otimes u^*) = \mu \Delta u^* - \nabla \Pi^* + \rho_\infty F(x),\\
        &\dive u^*=0,
    \end{aligned}
    \right.
\end{equation}  
where $x\in\mathbb{R}^3$, $u^*(x)=(u^*_1,u^*_2,u^*_3)(x)$ is the velocity of the incompressible fluid and $\Pi^*(x)$ is the pressure.
Under the assumption of the initial density $\rho_{\ep,0}-\rho_\infty=O(\ep)$ as $\ep\to 0$, that is $\ep^{-1}(\rho_{\ep,0}-\rho_\infty)$ is bounded in the function space $\Bi{1/2}\cap \dot{H}^{3}$ with respect to $0<\ep\ll 1$, we show that if the initial perturbation $(\rho_{\ep,0}-\rho^*_\ep,v_{0,\ep}-v^*_\ep)$ is small in the function space $\Bi{1/2}\cap \dot{H}^3$, then the perturbation $(\rho_\ep-\rho^*_\ep,v_\ep-v^*_\ep)$ exists globally in time such that
{ 
\begin{align*}
    \|\rho_\ep-\rho^*_\ep\|_{L^r_t(0,\infty;\dot{B}^{s}_{p,1})} \lesssim \max \left(\epsilon^{1+\frac{1}{r}}, \epsilon^{\frac{3}{2}-\frac{1}{p}}\right)
\end{align*}
and
\begin{align*}
    \|(I-\mathbb{P})(v_\ep-v^*_\ep)\|_{L^r_t(0,\infty;\dot{B}^{s}_{p,1})} + \|\mathbb{P}(v_\ep-v^*_\ep)-\tilde{u}_\ep\|_{L^r_t(0,\infty,\dot{B}^{s}_{p,1})} \lesssim \max \left(\epsilon^{\frac{1}{r}}, \epsilon^{\frac{1}{2}-\frac{1}{p}}\right)
\end{align*}
for any $0<\epsilon\ll 1$,} where { $2<p<\infty$, $2<r\leq \infty$ with $1/2+2/r<s<3/p$ and} $u_\ep$ is the global solution of the Cauchy problem of the incompressible Navier-Stokes equation:
\begin{equation} \label{ieq}
    \left\{ \,
    \begin{aligned}
        &\rho_\infty(\partial_{t}u_\ep + \dive(u_\ep \otimes u_\ep)) = \mu \Delta u_\ep - \nabla \Pi + \rho_\infty F(x),\\
        &\dive u_\ep=0,\\
        &u_\ep|_{t=0}=\mathbb{P}v_{0,\ep},
    \end{aligned}
    \right.
\end{equation}
where $t\geq 0$, $x\in\mathbb{R}^3$, $u_\ep(t,x)=(u_{\ep,1},u_{\ep,2},u_{\ep,3})(t,x)$ is the velocity, $\Pi(t,x)$ is the pressure and $\mathbb{P}$ is the Helmholtz projection. Here $\dot{B}^{s}_{p,1}$ and $\dot{B}^{s}_{p,\infty}$ denote the homogeneous Besov spaces and $\dot{H}^s$ denotes the homogeneous Sobolev space whose definitions are given in Section 2 below.

{   The mathematical analysis of the low Mach number limit for compressible flows was initiated by Klainerman and Majda \cite{MR615627}, \cite{zbMATH03752497}. They proved the local-in-time convergence of solutions to the compressible Euler equation towards the incompressible Euler equation, respectively, within the framework of the inhomogeneous Sobolev space. They also prove the local-in-time convergence result for a solution of the Navier-Stokes equation.
Furthermore, Isozaki \cite{zbMATH04003778}, \cite{zbMATH04179956} and Ukai \cite{zbMATH04003779} obtained results on the low Mach number limit for the compressible Euler equation by utilizing dispersive estimates for solutions of the linearized compressible Euler equation.} 
The low Mach number limit of the compressible viscous flow with ill-prepared initial data was studied by Lions-Masmoudi \cite{MR1628173} in the framework of weak solutions with large initial data and the pressure $P(\rho_\ep)=a \rho_\ep^\gamma$ where $a$, $\gamma$ are positive constants. It was shown { in \cite{MR1628173}} that any weakly convergent subsequence of a velocity field $v_\ep$ of the compressible flow converges to the velocity field $u$ of the corresponding incompressible flow. We also mention that the low Mach number limit of the stationary problem of the compressible Navier-Stokes equation was studied in \cite{MR1628173} for a bounded domain case. {  The low Mach number limit problem for the strong solution of the stationary compressible Navier-Stokes equation was studied in a bounded domain case by Beir\~ao da Veiga \cite{MR880415}, \cite{MR929289}.}
{ In \cite{MR1779621}, Danchin established the global existence of strong solutions for the compressible Navier-Stokes equation in $\mathbb{R}^N$ ($N \geq 2$) under a critical framework. Furthermore, in a subsequent work \cite{MR1886005}, it was shown that the strong solution to the compressible Navier-Stokes equation with the Mach number $\epsilon>0$ converges to the solution of the incompressible Navier-Stokes equation as $\ep \to 0$. This convergence holds with respect to the globally in time norm for ill-prepared initial data, under the assumption that a time-dependent force $F(t,x) \in L^1_t(0,\infty; \dot{B}_{2,1}^{N/2-1})$. This result was achieved by effectively applying the Strichartz estimates 
\begin{align} \label{oristr}
    \|(b,d)\|_{L^r_t(0,\infty;\dot{B}_{p,1}^{s+N\left(\frac{1}{p}-\frac{1}{2}\right)+\frac{1}{r}})}\lesssim \ep^{\frac{1}{r}} (\|(b_0.d_0)\|_{\dot{B}_{2,1}^s}+\|(f,g)\|_{L^1_t(0,\infty;\Bo{s})})
\end{align}
for a solution of the linear system
\begin{equation} \label{bdeq}
    \left\{ \,
    \begin{aligned}
        &\p_tb + \ep^{-1}(-\Delta)^{\frac{1}{2}} d = f, \\
        & \p_td - \ep^{-1}(-\Delta)^{\frac{1}{2}} b= g,\\
        &(b,d)|_{t=0}=(b_0,d_0),
    \end{aligned}
    \right.
\end{equation}
where $(-\Delta)^{1/2}$ denotes the Fourier multiplier with the symbol $|\xi|,\  \xi\in\mathbb{R}^N$.
Here $p\geq 2$, $2/r\leq \min\{1,(N-1)(1/2-1/p)\}$ and $(r,p,N)\neq (2,\infty,3)$. }The strong convergence results in the scaling critical spaces were also studied in \cite{MR2768550}, \cite{MR3563240} and \cite{MR4721788}. For the non-isentropic case, the low Mach number limit was investigated by \cite{MR2211706} and the case of magnetohydrodynamic in \cite{MR3197662}. { The review papers \cite{MR2157145}, \cite{MR3381190} give good surveys for the results for the low Mach number limit.}
So far, the low Mach number limit of time global strong solutions has been established around motionless states. In this paper, we investigate the low Mach number limit of strong solutions globally in time around a stationary solution with a nontrivial velocity field. 

We first present our result for the stationary problem (\ref{stationary_eq}).  The following theorem shows the existence and the low Mach number limit of the stationary solution of (\ref{stationary_eq}).

\begin{thm} \label{stthm}
    There exists a constant $\delta_0>0$ such that if $\|F\|_{\dot{B}^{-3/2}_{2,\infty}\cap\dot{H}^3}\leq \delta_0$,  then the stationary problems $(\ref{istationary_eq})$ and $(\ref{stationary_eq})$ with Mach number $0<\ep \leq 1$ have unique solutions $u^*$ and $(\rho_\epsilon^*,v_\epsilon^*)$ satisfying
    \begin{align} \label{stafor}
            \|u^*\|_{\Bi{\frac{1}{2}}} + \ep^{-1}\|\rho^*_\ep-\rho_\infty\|_{\Bi{-\frac{1}{2}}\cap \dot{H}^{4}} + \|v_\epsilon^*\|_{\Bi{\frac{1}{2}}\cap \dot{H}^{5}} \lesssim \|F\|_{\dot{B}^{-\frac{2}{3}}_{2,\infty}\cap\dot{H}^3}.
    \end{align}
    \begin{align}
        \ep^{-1} \|\rho^*_\ep-\rho_\infty\|_{\dot{B}^{-\frac{1}{2}}_{2,\infty}} + \|(\mathbb{Q}v_\epsilon^*, \mathbb{P}v^*_\epsilon -u^*)\|_{\dot{B}^{\frac{1}{2}}_{2,\infty}} \lesssim \epsilon\ \ \ \mathrm{for}\ \ \ \epsilon \ll 1,
    \end{align}
    where $\mathbb{P}$ is the Helmholtz projection and $\mathbb{Q}=I-\mathbb{P}$.
\end{thm}

Now, we state our main theorem, which derives the low Mach number limit for the perturbation of the stationary solution obtained in Theorem \ref{stthm}.

\begin{thm} \label{nstthm}
    Let $u^*$, $(\rho^*_\ep,v^*_\ep)$ be the stationary solutions satisfying $(\ref{stafor})$ with $\|F\|_{\Bi{-3/2}\cap\dot{H}^3}$ sufficiently small. Then, there exist constants $\delta_1, \epsilon_1>0$ such that if $\epsilon \leq \epsilon_1$, $\|F\|_{\Bi{-3/2}\cap\dot{H}^3}\leq \delta_1$  and the initial perturbations 
    \begin{align*}
        \mathbb{P}v_{\ep,0}-u^*\in \dot{B}^{\frac{1}{2}}_{2,\infty},\ \ \ (\rho_{\epsilon,0}-\rho^*_{\epsilon},v_{\epsilon, 0}-v^*_\epsilon)\in \Bi{\frac{1}{2}}\cap \dot{H}^4
    \end{align*}
    and
    \begin{align} \label{smallassumption2}
        \|\mathbb{P}v_{\ep,0}-u^*\|_{\dot{B}^{\frac{1}{2}}_{2,\infty}} + \ep^{-1} \|\rho_{\epsilon,0}-\rho^*_{\epsilon}\|_{\Bi{\frac{1}{2}}\cap \dot{H}^4} +\|v_{\epsilon, 0}-v^*_\epsilon\|_{\Bi{\frac{1}{2}}\cap \dot{H}^4}  \leq \delta_1,
    \end{align}
    then the problems $(\ref{eq})$ and $(\ref{ieq})$ have unique global solutions  $(\rho_\ep,v_\ep)$ with  $(\rho_\ep-\rho^*_\ep,v_\ep-v^*_\ep)\in C^0([0,\infty);\Bi{1/2}\cap\dot{H}^4)$ and $u_\ep\in C^0([0,\infty);\Bi{1/2})$, respectively, satisfying $(\rho_\ep-\rho^*_\ep,v_\ep-v^*_\ep)\in C^0([0,\infty);\Bi{1/2}\cap\dot{H}^4)$ and
    \begin{align} \label{inifor2}
        &\sup_{0\leq t <\infty}\|u_\ep(t)-u^*\|_{\Bi{\frac{1}{2}}} + \ep^{-1}\sup_{0\leq t<\infty} \|\rho_\ep(t)-\rho^*_\ep\|_{\Bi{\frac{1}{2}}\cap \dot{H}^4} \nonumber \\
        &\hspace{150pt} + \sup_{0\leq t<\infty} \|v_\ep(t)-v^*_\ep\|_{\Bi{\frac{1}{2}}\cap \dot{H}^4} \lesssim \delta_1.
    \end{align}
    Furthermore, the following low Mach number limit result holds.
    Let $2<p<\infty$, $2<r\leq \infty$ and $1/2+2/r<s<3/p$. Then, we have
    \begin{align} \label{machesti}
        &\ep^{-1}\|\rho_\ep-\rho^*_\ep\|_{L^r_t(0,\infty;\dot{B}^{s}_{p,1})} +\|\mathbb{Q}(v_\ep-v^*_\ep)\|_{L^r_t(0,\infty;\dot{B}^{s}_{p,1})} \nonumber \\
        &\hspace{120pt} + \|\mathbb{P}w_\epsilon-\tilde{u}_\ep\|_{L^r_t(0,\infty,\dot{B}^{s}_{p,1})} \lesssim \max \left(\epsilon^{\frac{1}{r}}, \epsilon^{\frac{1}{2}-\frac{1}{p}}\right)\delta_1,
    \end{align}
    where $\mathbb{Q}=I-\mathbb{P}$, $w_\ep = v_\ep-v_\ep^{*}$ and $\tilde{u}_\ep=u_\ep-u^*$.
\end{thm}

By the continuous inclusion $\dot{B}_{q,1}^{3(1/q-1/p)} \hookrightarrow L^{p}$, $1\leq q\leq p\leq \infty$ (see \cite[Proposition 2.39]{MR2768550} for example ), we obtain the following corollary.

\begin{cor}
    Under the same assumption as in Theorem $\ref{nstthm}$, if $3<p<\infty$, $2<r\leq \infty$ with $2/r<1-3/p$, then we have
    \begin{align*}
        &\ep^{-1}\|\rho_\ep-\rho^*_\ep\|_{L^r_t(0,\infty;L^p)} +\|\mathbb{Q}(v_\ep-v^*_\ep)\|_{L^r_t(0,\infty;L^p)} \nonumber \\
        &\hspace{120pt} + \|\mathbb{P}w_\epsilon-\tilde{u}_\ep\|_{L^r_t(0,\infty;L^p)} \lesssim \max \left(\epsilon^{\frac{1}{r}}, \epsilon^{\frac{1}{2}-\frac{1}{p}}\right).
    \end{align*}
\end{cor}

{  The employment of Besov spaces is motivated by the following two reasons. The first reason is the slow decay of the stationary velocity field $v^{*}_\ep(x)$ as $|x|\to\infty$. The spatial decay order of $v^*_{\ep}$ is expected at most $|v^{*}_{\ep}|=O(1/|x|)$ as $x\to\infty$ by comparison with the results for the incompressible Navier-Stokes equation \cite{MR1794517} and \cite{zbMATH05883984}. Since the function $1/|x|$ belongs to the Besov space $\Bi{1/2}$, we employ the Besov space as a framework for constructing the stationary solution. We mention that the function $1/|x|$ does not belong to the Besov space $\Bo{1/2}$, which is used in \cite{MR1886005} as the function space for an initial velocity in the three dimensional case. The second reason is that we aim to choose a class of function spaces for the perturbations around the stationary solution that also includes the stationary solution. Indeed, in Theorem \ref{nstthm}, we prove the existence and the low Mach number limit of solutions to (\ref{eq}) in the $\dot{B}_{2,\infty}^{1/2} \cap \dot{H}^4$-framework, and the space $\dot{B}_{2,\infty}^{1/2} \cap \dot{H}^4$ contains the stationary solution $(\rho^*_\ep-\rho_\infty,v^*_\ep)$ obtained in Theorem \ref{stthm}.}

{  The proofs of the existence of stationary and non-stationary solutions in Theorems \ref{stthm} and \ref{nstthm}, respectively, are based on the approach taken in \cite{MR4803477}. However, since the paper \cite{MR4803477} only addressed the special case where the Mach number $\epsilon=1$, we performed a more precise analysis that incorporates the Mach number dependence to prove the existence of results in Theorems \ref{staexis} and \ref{time-decay} below. In Theorem \ref{time-decay}, the existence result for the non-stationary problem in the $\Bi{1/2}\cap\dot{H}^4$-framework. The $\dot{H}^4$ norm estimate for the solutions will be applied to the estimate $(\ref{4reason})$ in the proof of Theorem \ref{nstthm} below. }

{ The difficulty in deriving the low Mach number limit for the perturbations around the nontrivial stationary solution arises from the slow spatial decay of the stationary solution at spatial infinity. The spatial decay rate of the stationary solution at infinity is at most $1/|x|$, which implies that the linear terms ${v}^*_\ep \cdot \nabla ({v_\ep} - {v}^*_\ep)$ and $({v_\ep} - {v}^*_\ep)\cdot\nabla v_\ep^*$, which appear in the perturbation equation, are not expected to belong to $L^1$ in time. Consequently, it seems difficult to directly apply the Strichartz estimates $(\ref{oristr})$, such as those used in \cite{MR3563240}, \cite{MR4721788}, where the inhomogeneous terms $(f,g)$ are assumed to be $L^1$ in time, suggesting the necessity of new type estimates. To overcome this difficulty, we derive the Strichartz type estimates for the semigroup $e^{tA_{\ep}}$ which reflect not only its dispersive property but also dissipative properties. Here $e^{tA_\epsilon}$ denotes the semigroup generated by
generated by
\begin{align*}
    A_\ep = 
    \begin{bmatrix}
        0 & -\ep^{-1}\dive\\
        -\ep^{-1} \nabla  &  (2\mu+\mu')\nabla \dive
    \end{bmatrix}.
\end{align*} 
We show the following Strichartz type estimate for the inhomogeneous term in Proposition \ref{strichartztypeestimate} below:
\begin{align} \label{st2}
    \left\|\int_0^t e^{\tau A_\ep}\Psi(t-\tau)d\tau \right\|_{L^r_t(0,\infty; \dot{B}_{p,1}^{s})}  \lesssim_{p,r} \ep^{1-\frac{2}{p}} \|\Psi\|_{L^r_t(0,\infty; \dot{B}^{s+2-\frac{8}{p}}_{p',1}\cap \dot{B}^{s+\frac{3}{2}-\frac{3}{p}}_{2,1})},
\end{align}
where $2\leq p<\infty$, $2<r\leq \infty$ and $s\in\mathbb{R}$. The Strichartz estimate (\ref{oristr}), which was utilized in \cite{MR2768550}, \cite{MR1855277} and \cite{MR4721788}, derives the $L^r$ in time estimate ($r\geq 2$) for the solution of $(\ref{bdeq})$ when the source term $(f,g)$ belongs to the function space that is $L^{1}$ in time. We note that a more general Strichartz estimate for \eqref{bdeq} was established in \cite[Proposition 10.30]{MR2768550}:
\begin{align} \label{genearlizedstr}
\|(b,d)\|_{\widetilde{L}^r
(0,\infty;\dot{B}_{p,1}^{s+N\left(\frac1p-\frac12\right)+\frac1r}(\mathbb{R}^N))}
&\le
C\varepsilon^{\frac1r}
\|(b_0,d_0)\|_{\dot{B}_{2,1}^{s}(\mathbb{R}^N)}
\nonumber\\
&\quad
+\varepsilon^{1+\frac1r-\frac1{\bar r}}
\|(f,g)\|_{
\widetilde{L}^{\bar r'}
(0,\infty;
\dot{B}_{ p',1}^{\,s+N\left(\frac1{\bar p}-\frac12\right)+\frac1{\bar r}}
(\mathbb{R}^N))
}
\end{align}
with
\begin{align}
p&\ge2,
&
\frac2r&\le\min(1,\gamma(p)),
&
(r,p,N)&\neq(2,\infty,3), \label{aaa}
\\ 
\bar p&\ge2,
&
\frac2{\bar r}&\le\min(1,\gamma(\bar p)),
&
(\bar r,\bar p,N)&\neq(2,\infty,3), \label{bbb}
\end{align}
where $N\geq 2$, $\gamma(q):=(N-1)(1/2-1/q)$, $1/\bar p + 1/\bar p '=1$, $1/\bar r + 1/\bar r '=1$. Here, $\widetilde{L}^r
(\dot{B}_{p,1}^{s+N(1/p-1/2)+1/r})$, $\widetilde{L}^{\bar r'}
(
\dot{B}_{ p',1}^{\,s+N(1/{\bar p}-1/2)+1/{\bar r}}
)$ are the Chemin-Lerner spaces. Even if one employs the generalized Strichartz estimate \eqref{genearlizedstr}, one still cannot derive a $L_t^r$--$L_t^r$ estimate from the source term $(f,g)$ to the corresponding solution $(b,d)$ in the three-dimensional case $N=3$, owing to the restrictions \eqref{aaa} and \eqref{bbb} on the exponents $(p,\bar p,r,\bar r)$. In contrast, the Strichartz type estimate (\ref{st2}) shows the $L^r$ in time estimate ($r>2$) for the inhomogeneous term when the source term $\Phi$ belongs to the function space that is $L^r$ in time. The advantage of the estimate $(\ref{st2})$ is that it allows us to close the estimate for the inhomogeneous term whose source term contains the term with the non-time decaying coefficient, such as the terms $v^*_\ep \cdot\nabla w_\ep$, $(\dive\,v^*_\ep) w_\ep$. From the viewpoint of perturbation theory, the difficulty arises from the fact that the terms ${v}^*_\ep \cdot \nabla ({v_\ep} - {v}^*_\ep)$ and $({v_\ep} - {v}^*_\ep)\cdot\nabla v_\ep^*$ cannot be treated as a simple perturbation of the operator $A_\ep$ (or the semigroup $e^{tA_\ep}$), due to the slow decay of $v^{*}_\ep$. In \cite{MR4803477}, the perturbation analysis was developed in view of the time-decay estimate.  However, in this paper, it is necessary to conduct this perturbation analysis considering not only the time decay but also the Mach number dependence. For this purpose, the Strichartz estimates (\ref{oristr}) using only the dispersive properties are insufficient; we employ the Strichartz-type estimate (\ref{st2}), which also reflects the dissipative properties of the operator $A_\ep$.

The derivation of the estimate $(\ref{st2})$ is carried out by separating the effects of the heat semigroup $e^{(\mu+\mu'/2)t\Delta}$ and the dispersive semigroup $e^{it\frac{|\nabla|}{\ep}}$ from the semigroup $e^{tA_\ep}$. The key step of the separation is to derive the exponentially growing estimate
\begin{align} \label{a}
    \|\delj\mathcal{F}^{-1}[ e^{\lambda_{\pm}t}\hat{\psi}]\|_{L^p}\lesssim e^{\kappa \ep 2^{3j}t}\|\delj e^{(\mu+\mu'/2)t\Delta}e^{\pm t\frac{|\nabla|}{\ep}}\psi\|_{L^p}
\end{align}
when $\ep2^{j}>0$ is small, where $j\in\mathbb{Z}$, $1< p<\infty$, $\kappa>0$ is a constant, $\delj$ is the dyadic block, which will be introduced in Section 2, and
\begin{align*}
    \lambda_{\pm}(\xi) = -\mu_0|\xi|^{2} \pm \sqrt{\mu_0^{2}|\xi|^{4}-\ep^{-2} |\xi|^{2}}, \ \ \ \xi\in\mathbb{R}^{3}
\end{align*}
are the eigenvalues of the matrix 
\begin{align*}
    \hat{A_\ep}(\xi) = 
    \begin{bmatrix}
        0 & -i\ep^{-1}\xi^{\mathsf{T}}\\
        -i\ep^{-1} \xi & -2\mu_0 \xi \otimes \xi
    \end{bmatrix},\ \ \ \xi\in\mathbb{R}^3,
\end{align*}
where $\mu_0 = \mu + \mu'/2$. If $\ep2^j>0$ is small enough, the exponential growth $e^{\kappa\ep 2^{3j}}$ can be controlled by the estimate of the heat semigroup:
\begin{align} \label{b}
    \|\delj e^{(\mu/2+\mu'/4)t\Delta}\psi\|_{L^p} \lesssim e^{-c2^{2j}t}\|\delj \psi\|_{L^p}\ \ \ \mathrm{for \ any}\ \ \ j\in\mathbb{Z}, \ t\geq 0
\end{align}
which is stated in Lemma \ref{heatj} below. The proof of the estimate (\ref{a}) is carried out by the Taylor expansion at $\xi=0$ of the error multiplier $e^{t(\lambda_{\pm}(\xi)-\lambda_{\pm,0}(\xi))}$, where $\lambda_{\pm,0}=- \left(\mu+ \mu'/2 \right) |\xi|^2 \pm i|\xi|/\ep$. If $\ep2^j>0$ is small enough, then the estimates (\ref{a}) and (\ref{b}) provide the separating estimate
\begin{align} \label{c}
    \|\delj e^{tA}\psi\|_{L^{p}} \lesssim e^{-c2^{2j}t}\|\delj e^{\pm i t \frac{|\nabla|}{\ep}}\|_{L^p},
\end{align}
where $c>0$ is a constant. The detailed proof for the estimate (\ref{c}) will be carried out in the proof of Lemma \ref{approeA} below. Then, the proof of (\ref{st2}) is completed by combining (\ref{c}) with the dispersive estimate for the semigroup $e^{\pm i t|\nabla|/\ep}$. The dispersive estimate (Lemma \ref{linesti0}) relies on the assumption that the domain is the whole space. 

The outline of this paper is as follows: In Section 2, we introduce the homogeneous Sobolev and Besov spaces and present several lemmas to be used in the subsequent sections. Section 3 is devoted to the proof of Theorem \ref{stthm}. The strategy for proving the existence of the solution to the stationary problem (\ref{stationary_eq}) is based on the approach used in \cite{MR4803477}. More precisely, the existence of the stationary solution is proved by demonstrating that the approximate operator $\Phi_{j,\ep}$ (defined in Lemma \ref{lip} below) is a contraction mapping in a neighborhood of origin in $\Bi{-1/2}\times \Bi{1/2}$, and is bounded in a neighborhood of origin in $(\Bi{-1/2}\cap\dot{H}^5)\times (\Bi{1/2}\cap\dot{H}^6)$. In Lemma \ref{lip} below, we show the uniform boundedness and contractivity estimates for the operator $\Phi_{j,\ep}$ with respect to small Mach numbers $\ep>0$. Since the stationary solution in the whole space $\mathbb{R}^3$ can only be expected to have slow decay, specifically $|v^*_{\epsilon}(x)|=O(1/|x|)\ \mathrm{as}\  |x|\to\infty$, we are required to perform our analysis using the function space with low-frequency constraints compared to the results obtained by Lions-Masmoudi \cite{MR1628173}, Beir\~ao da Veiga \cite{MR880415}, \cite{MR929289} for the bounded domain case. Section 4 is devoted to the proof of Theorem \ref{nstthm}. First, in Theorem \ref{time-decay} , we extend the global existence result in $\Bi{1/2}\cap\dot{H}^3$ framework to $\Bi{1/2}\cap\dot{H}^4$ framework for our analysis. Next, we show the Strichartz type estimate for the semigroup $e^{tA_\ep}$, which is a key estimate in the proof. Finally, we establish the proof of Theorem \ref{nstthm} by applying the Stricharz type estimate and the bilinear estimates in Besov spaces in Lemmas \ref{bi1}, \ref{bi2} below.
} 

\begin{nota}
    The notation $A\lesssim_{\alpha} B$ means that there exists a constant $C$ depending on $\alpha$ such that $A\leq CB$. The notation $A\sim_\alpha B$ means that $A\lesssim_\alpha B$ and $B\lesssim_\alpha A$. We denote a commutator by $[X,Y]= XY-YX$. We write $\mathcal{S}$ for the set of all Schwartz functions on $\mathbb{R}^{3}$, and we write $\mathcal{S}'$ for the set of all tempered distributions on $\mathbb{R}^{3}$. The notations $\hat{\cdot}$ and $\mathcal{F}$ stand for the Fourier transform
    \begin{align*}
        \hat{u}(\xi) = \mathcal{F}(u)(\xi)= \int_{\mathbb{R}^{3}} e^{-i x \cdot \xi} u(x)dx,
    \end{align*}
    and the notation $\mathcal{F}^{-1}$ denotes the inverse Fourier transform. The symbol $\mathbb{P}$ denotes the Helmholtz projection: $\mathbb{P}u = u- \Delta^{-1}\nabla \dive\,u$, $u\in \mathcal{S}'$, and the symbol $\mathbb{Q}$ denotes $\mathbb{Q}=I-\mathbb{P}$. We denote the $L^{2}(\mathbb{R}^3)$ inner product by $\langle u,v\rangle = \int_{\mathbb{R}^3} u vdx$. For any Banach space $Z$ and $1\leq p\leq \infty$, we define the function space $L_t^p(0,\infty;Z)=L_t^p(Z)$ by the set of strongly measurable functions $f:(0,\infty)\to Z$ such that 
    \begin{align*}
        \|f\|_{L_t^p(0,\infty;Z)}=\|f\|_{L_t^p(Z)}=\| \|f(t)\|_Z \|_{L^{p}_t((0,\infty))}<\infty.
    \end{align*}

\end{nota}

\section{Preliminary}
This section introduces the homogeneous Sobolev and Besov spaces and presents several lemmas that are frequently used in this paper. For any $s\in\mathbb{R}$, the homogeneous Sobolev space $\dot{H}^{s}$ is defined as the set of tempered distributions $u \in \mathcal{S}'$ such that $\hat{u}\in L^{1}_{loc}(\mathbb{R}^3)$ and 
\begin{align*}
    \|u\|_{\dot{H}^{s}}:= \||\cdot|^{s}\hat{u}\|_{L^2}<\infty.
\end{align*}
We next define the homogeneous Besov space. {  Let $\mathcal{S}'_h$ be the subspace of $u\in\mathcal{S}'$ such that
\begin{align*}
    \|\mathcal{F}^{-1}(\theta(\lambda\cdot)\hat{u})\|_{L^\infty} \to 0\ \ \ \mathrm{as}\ \ \ \lambda\to \infty
\end{align*}
for any $\theta\in C^\infty_0(\mathbb{R}^3)$.
}
We employ the following squared dyadic partition of unity, which we use the proof of Thorem \ref{staexis}. We fix $\phi\in C^{\infty}(\mathbb{R}^3)$ supported in the annulus $\mathcal{C}=\{\xi\in \mathbb{R}^3 \mid 3/4 \leq |\xi| \leq 8/3\}$ such that
\begin{align*}
    \sum_{j\in\mathbb{Z}}\phi^2(2^{-j}\xi)=1\ \ \ \mathrm{for}\ \ \ \xi\neq 0.
\end{align*}
Define the dyadic blocks $(\delj)_{j\in\mathbb{Z}}$ by the Fourier multiplier
\begin{align*}
    \delj u = \mathcal{F}^{-1} [\phi^2(2^{-j}\cdot) \hat{u}],\ \ \ { u\in\mathcal{S}'}.
\end{align*}
The homogeneous low-frequency cutoff operator is denoted by
\begin{align} \label{lowfreqcut}
    \dot{S}_{j}u = \sum_{j' < j}\deljj u,\ \ j\in \mathbb{Z},\ { u\in\mathcal{S}'}.
\end{align}
The homogeneous Besov space is defined as follows. Let $s\in\mathbb{R}$, $1\leq p, r \leq \infty$. Then, the homogeneous Besov space $\dot{B}^{s}_{p,r}=\dot{B}^{s}_{p,r}(\mathbb{R}^{3})$ is given by
\begin{align*}
    \dot{B}^{s}_{p,r}= \left\{ u \in \mathcal{S}'_h\  \middle|\  \|u\|_{\dot{B}^{s}_{p,r}}=  \left\|(2^{js}\|\delj u\|_{L^p})_{j\in\mathbb{Z}}\right\|_{\ell^{r}} < \infty  \right\}.
\end{align*}

{  
The following embedding property is proved in \cite[Proposition 2.20]{MR2768550}.
\begin{prop} \label{emb}
    Let $1\leq p_1\leq p_2\leq \infty$, $1\leq r_1\leq r_2\leq \infty$ and $s\in\mathbb{R}$. Then, the space $\dot{B}^{s}_{p_1,r_1}$ is continuously embedded in the space $\dot{B}^{s-3(1/p_1-1/p_2)}_{p_2,r_2}$.
\end{prop}
We derive several bilinear estimates in Besov spaces. We first recall the Bony decomposition of the product $uv$. Let $u, v \in \mathcal{S}'_h$. Then, at least formally, the product $uv$ can be decomposed as
\begin{align*}
    uv=\dot{T}_u v+\dot{T}_v u +\dot{R}(u,v)
\end{align*}
where the paraproduct $\dot{T}_u v$ and the remainder $\dot{R}(u, v)$ are defined by
\begin{align*}
    \dot{T}_u v =\sum_{j\in\mathbb{Z}}\dot{S}_j u \delj v
\end{align*}
and
\begin{align*}
    \dot{R}(u,v)= \sum _{|j-j'|\leq 1} \delj u\ \dot{\Delta}_{j'}v.
\end{align*}
The following continuity property for the paraproduct operator holds. 
\begin{prop} \label{pararem1}
    Let $1\leq p,p_1,p_2\leq \infty$, $1\leq r,r_1, r_2\leq \infty$, $s\in\mathbb{R}$ and $t<0$. If $1/p=1/p_1+1/p_2$ and $1/r\leq 1/r_1+1/r_2$, then we have
    \begin{align} \label{paraeq1}
        \|\dot{T}_u v\|_{\dot{B}^{s+t}_{p,r}} \lesssim \|u\|_{\dot{B}^{t}_{p_1,r_1}} \|v\|_{\dot{B}^s_{p_2,r_2}}
    \end{align}
    for any $u\in \dot{B}^{t}_{p_1,r_1}$ and $v\in \dot{B}^{s}_{p_2,r_2}$. 
\end{prop}
\begin{proof}
    Since there exists an annulus $\tilde{\mathcal{C}}$ such that the support of $\mathcal{F}(\dot{S}_j u \delj v)$ is contained in $2^j \tilde{\mathcal{C}}$ for all $j\in\mathbb{Z}$, it is sufficient to prove that
    \begin{align} \label{moku}
        \|\{2^{j(s+t)}\|\dot{S}_j u \delj v\|_{L^p}\}_{j\in\mathbb{Z}}\|_{\ell^r(\mathbb{Z})} \lesssim 1
    \end{align} 
    for any  $u\in \dot{B}^{t}_{p_1,r_1}$ and $v\in \dot{B}^{s}_{p_2,r_2}$ with $\|u\|_{\dot{B}^{t}_{p_1,r_1}}= \|v\|_{\dot{B}^s_{p_2,r_2}}=1$. 
    By taking $\ell^{r_1}(\mathbb{Z})$ norm for both side of the inequality
    \begin{align*}
        2^{jt}\|\dot{S}_j u\|_{L^{p_1}} \leq \sum_{j'<j} 2^{(j-j')t} 2^{j't}\|\dot{\Delta}_{j'}u\|_{L^{p_1}},
    \end{align*}
    Young's inequality shows that
    \begin{align*}
        \|\{2^{jt}\|\dot{S}_j u\|_{L^{p_1}}\}_{j\in\mathbb{Z}}\|_{\ell^{r_1}(\mathbb{Z})} \lesssim 1.
    \end{align*}
    Thus, H\"{o}lder's inequality leads (\ref{moku}).
\end{proof}
The following continuity property for the remainder is proved in \cite[Theorem 2.52]{MR2768550}.
\begin{prop} \label{pararem2}
    Let $s_1,s_2\in \mathbb{R}$ and $1\leq p, p_1,p_2,r,r_1,r_2\leq \infty$. If $1/p= 1/p_1+1/p_2$ and $1/r\leq 1/r_1+1/r_2$ and $s_1+s_2>0$, then
    \begin{align*}
        \|\dot{R}(u,v)\|_{\dot{B}^{s_1+s_2}_{p,r}} \lesssim \|u\|_{\dot{B}^{s_1}_{p_1,r_1}}\|v\|_{\dot{B}_{p_2,r_2}^{s_2}}.
    \end{align*}
    for any $u\in\dot{B}^{s_1}_{p_1,r_1}$ and $v\in\dot{B}_{p_2,r_2}^{s_2}$.
    Moreover, if $s_1+s_2\geq 0$ and $ 1/r_1+r_2\geq 1$, then we have
    \begin{align*}
        \|\dot{R}(u,v)\|_{\dot{B}^{s_1+s_2}_{p,\infty}} \lesssim \|u\|_{\dot{B}^{s_1}_{p_1,r_1}}\|v\|_{\dot{B}_{p_2,r_2}^{s_2}}
    \end{align*}
    for any $u\in\dot{B}^{s_1}_{p_1,r_1}$ and $v\in\dot{B}_{p_2,r_2}^{s_2}$.
\end{prop}
Now, Propositions \ref{pararem1}, \ref{pararem2} lead to the following bilinear estimates in Besov spaces.
}

\begin{lem} \label{bi1}
    Let $2\leq p\leq \infty$, $1\leq r,r_1,r_2\leq \infty$ and $s_1,s_2<3/p$ with $s_1+s_2>0$. If $1/r\leq 1/r_1+1/r_2$, then for any $u\in \dot{B}^{s_1}_{2,r_1}$, $v\in \dot{B}^{s_2}_{2,r_2}$, we have $uv\in \dot{B}^{s_1+s_2-\frac{3}{p}}_{p',r}$ and
    \begin{align} \label{pp}
        \|uv\|_{\dot{B}^{s_1+s_2-\frac{3}{p}}_{p',r}} \lesssim \|u\|_{\dot{B}^{s_1}_{2,r_1}} \|v\|_{\dot{B}^{s_2}_{2,r_2}},
    \end{align}
    where $1/p':=1-1/p$.
\end{lem}
{ 
\begin{proof}
    Let $1/q:=1/2-1/p$. Then, by Propositions \ref{emb}, \ref{pararem1}, we have
    \begin{align*}
        \|\dot{T}_u v\|_{\dot{B}^{s_1+s_2-\frac{3}{p}}_{p',r}} \lesssim \|u\|_{\dot{B}_{q,r_1}^{s_1-\frac{3}{p}}} \|v\|_{\dot{B}_{2,r_2}^{s_2}} \lesssim \|u\|_{\dot{B}^{s_1}_{2,r_1}}\|v\|_{\dot{B}_{2,r_2}^{s_2}} 
    \end{align*}
    and
    \begin{align*}
        \|\dot{T}_v u\|_{\dot{B}^{s_1+s_2-\frac{3}{p}}_{p',r}} \lesssim  \|v\|_{\dot{B}^{s_2}_{2,r_2}}\|u\|_{\dot{B}_{2,r_1}^{s_1}}.
    \end{align*}
    By Propositions \ref{emb}, \ref{pararem2}, 
    \begin{align*}
        \|\dot{R}(u,v)\|_{\dot{B}^{s_1+s_2-\frac{3}{p}}_{p',r}} \lesssim \|\dot{R}(u,v)\|_{\dot{B}^{s_1+s_2}_{1,r}} \lesssim \|v\|_{\dot{B}^{s_1}_{2,r_1}}\|u\|_{\dot{B}_{2,r_2}^{s_2}}.
    \end{align*}
    Therefore, we obtain the estimate (\ref{pp}).
\end{proof}

}

\begin{lem} \label{bi2}
    Let $2\leq p\leq \infty$, $1\leq r,r_1,r_2\leq \infty$ and $s_1<3/p$, $s_2<3/2$ with $s_1+s_2>0$. If $1/r\leq 1/r_1+1/r_2$, then for any $u\in \dot{B}^{s_1}_{p,r_1}$, $v\in \dot{B}^{s_2}_{2,r_2}$, we have $uv\in \dot{B}^{s_1+s_2-\frac{3}{2}}_{p,r}$ and
    \begin{align} \label{are}
        \|uv\|_{\dot{B}^{s_1+s_2-\frac{3}{2}}_{p,r}} \lesssim \|u\|_{\dot{B}^{s_1}_{p,r_1}} \|v\|_{\dot{B}^{s_2}_{2,r_2}}.
    \end{align}
\end{lem}
{ 
\begin{proof}
    Propositions \ref{pararem1}, \ref{pararem2} imply that
    \begin{align*}
        \|\dot{T}_u v\|_{\dot{B}^{s_1+s_2-\frac{3}{p}}_{2,r}}\lesssim \|u\|_{\dot{B}^{s_1-\frac{3}{p}}_{\infty,r_1}}\|v\|_{\dot{B}_{2,r_2}^{s_2}},
    \end{align*}
    \begin{align*}
        \|\dot{T}_v u\|_{\dot{B}^{s_1+s_2-\frac{3}{2}}_{p,r}} \lesssim  \|v\|_{\dot{B}^{s_2-\frac{3}{2}}_{\infty,r_2}}\|u\|_{\dot{B}_{p,r_1}^{s_1}}
    \end{align*}
    and
    \begin{align*}
        \|\dot{R}(u,v)\|_{\dot{B}^{s_1+s_2}_{p',r}} \lesssim  \|v\|_{\dot{B}^{s_1}_{p,r_1}}\|u\|_{\dot{B}_{2,r_2}^{s_2}}.
    \end{align*}
    Therefore, Proposition \ref{emb} shows the estimate (\ref{are}).
\end{proof}

We use the following lemmas regarding the composition of functions.
\begin{lem} \label{comp1}
    Let $c>0$, $\Phi\in C^{\infty}(\mathbb{R})$ with $\Phi(0)=0$ and $\sigma\in\dot{B}^{s}_{2,r}\cap L^{\infty}$ with $s<2/3$, $1 \leq\infty$ or $s=3/2$, $r=1$ satisfies
    \begin{align*}
        \|\sigma\|_{L^\infty}\leq c.
    \end{align*}
    Then, we have
    \begin{align*}
        \|\Phi(\sigma)\|_{\dot{B}_{2,r}^s}\lesssim_{c,\Phi} \|\sigma\|_{\dot{B}_{2,r}^s}.
    \end{align*}
\end{lem}
As for the proof of Lemma \ref{comp1}, see \cite[Theorem 2.61]{MR2768550} for example.
\begin{lem} \label{compos}
    Let $c>0$, $\Phi\in C^{\infty}(\mathbb{R})$ and $\sigma,\eta\in \dot{B}^{s}_{2,r}\cap \dot{B}^{3/2}_{2,1}$ with $-3/2\leq s<3/2$, $1\leq r\leq \infty$ or $s=3/2$, $r=1$ satisfy
    \begin{align*}
        \|(\sigma,\eta)\|_{L^\infty} \leq c.
    \end{align*}
    Then, we have
        \begin{align} \label{2compesti}
            \|\Phi(\sigma)-\Phi(\eta)\|_{\dot{B}^{s}_{2,r}} \lesssim_{c,\Phi} (1+\|(\sigma,\eta)\|_{\Bo{\frac{3}{2}}}) \|\sigma-\eta\|_{\dot{B}^{s}_{2,r}}.
        \end{align}
\end{lem}
\begin{proof}
    Since
    \begin{align*}
        \Phi(\sigma)-\Phi(\eta)=\int_0^1 (\Phi'(\eta +t(\sigma-\eta))-\Phi'(0))(\sigma-\eta)dt + \Phi'(0)(\sigma-\eta),
    \end{align*}
    it follows from Proposition \ref{bi1} that
    \begin{align*}
        &\|\Phi(\sigma)-\Phi(\eta)\|_{\dot{B}^{s}_{2,r}}\\
        &\lesssim \sup_{0\leq t\leq 1}\|\Phi'(\eta+t(\sigma-\eta))-\Phi'(0)\|_{\dot{B}^{\frac{3}{2}}_{2,1}} \|\sigma-\eta\|_{\dot{B}^s_{2,r}} + |\Phi'(0)|\|\sigma-\eta\|_{\dot{B}^{s}_{2,r}}.
    \end{align*}
    By Proposition \ref{comp1},
    \begin{align*}
    \sup_{0\leq t\leq 1}\|\Phi'(\eta +t(\sigma-\eta))-\Phi'(0)\|_{\dot{B}^{\frac{3}{2}}_{2,1}}\lesssim_{c,\Phi} \|(\sigma,\eta)\|_{\dot{B}^{\frac{3}{2}}_{2,1}}.
    \end{align*}
    Therefore, we have the estimate (\ref{2compesti}).
\end{proof}
}

We use the following commutator estimates. 

\begin{lem} \label{commu}
    Let $-3/2<s<5/2$, $1\leq r\leq \infty$ and $\phi_0\in C^{\infty}_{0}(\mathbb{R}^3)$ with $\operatorname{supp}\phi_0\subset \mathcal{C}'$ for some annulus $\mathcal{C}'$ centered at the origin. Let us denote $\chi_{j}v=\mathcal{F}^{-1}\left[ \phi_0(2^{-j}\cdot)\hat{v} \right]$ for any $v\in\mathcal{S}'$, $j\in\mathbb{Z}$. Then, we have
    \begin{align*}
        \left\|\left(2^{js}\|[\chi_j, h\partial_k]u\|_{L^2}\right)_{j\in\mathbb{Z}}\right\|_{\ell^{r}(\mathbb{Z})} \lesssim_{s,\chi}\|\nabla h\|_{\dot{B}^{\frac{3}{2}}_{2,1}} \|u\|_{\dot{B}^{s}_{2,r}},
    \end{align*}
    where $1\leq k \leq 3$ and $u$, $h$ are scalar functions.
\end{lem}
{  The proof of Theorem \ref{commu} is obtained by replacing $\dot{\Delta}_{j}$ with $\chi_j$ in the proof of the homogeneous version of \cite[Lemma 2.100]{MR2768550}, as claimed in \cite[Remark 2.102]{MR2768550}.}

{ We show some estimates for the heat semigroup $e^{t\Delta}$ in Besov spaces. }

\begin{lem} \label{heatj}
    Let $1\leq p \leq \infty$ and $\psi \in \mathcal{S}'$. Then, for any $j\in\mathbb{Z}$, we have
    \begin{align*}
        \|\delj e^{t\Delta}\psi\|_{L^p} \lesssim e^{-c2^{2j}t}\|\delj \psi\|_{L^p}\ \ \ \mathrm{for \ any}\ \ \ t\geq 0,
    \end{align*}
    where $c>0$ is a constant.
\end{lem}
As for the proof of Lemma \ref{heatj}, see \cite[Lemma 2.4]{MR2768550} for example.

\begin{lem} \label{duah}
    Let $s\in \mathbb{R}$, $1\leq p,r \leq \infty$ and $h\in \mathcal{S}'$. Then, we have
    \begin{align*}
        \left\|\int_0^t e^{\tau\Delta}h(t-\tau)d\tau \right\|_{L^r_{t}(\dot{B}^{s}_{p,r})} \lesssim \|h\|_{L^r_t(\dot{B}^{s-2}_{p,r})}.
    \end{align*}
\end{lem}
{ \begin{proof}
    Let 
    \begin{align*}
        v(t)=\int_0^t e^{\tau\Delta}h(t-\tau)d\tau.
    \end{align*}
    Then, the following estimate follows from \cite[Corollary 2.5]{MR2768550}:
    \begin{align} \label{cor2.5}
        \|\delj v\|_{L^{r}_t(L^p)} \lesssim 2^{-2j}\|\delj h\|_{L^r_t(L^p)}\ \ \ \mathrm{for\ any}\ \ \ j\in\mathbb{Z}.
    \end{align}
    Bt taking $\ell^r(\mathbb{Z})$ norm for both side of (\ref{cor2.5}), we obtain
    \begin{align*}
        \left\|v \right\|_{L^r_{t}(\dot{B}^{s}_{p,r})} \lesssim \|h\|_{L^r_t(\dot{B}^{s-2}_{p,r})}.
    \end{align*}
\end{proof}}
{ The following lemma provides the dispersive estimate for the semigroup $e^{\pm\frac{|\nabla|}{\ep}t}$.}
\begin{lem} \label{linesti0}
    Let $2\leq p\leq \infty$, $j\in\mathbb{Z}$ and $\psi\in \mathcal{S}'$. Then, we have
    \begin{align} \label{j}
        \|\delj e^{\pm i\frac{|\nabla|}{\ep}t}\psi\|_{L^p} \lesssim 2^{2j\left( 1-\frac{2}{p}\right)}\ep^{1-\frac{2}{p}} |t|^{-\left( 1-\frac{2}{p}\right)}\|\delj \psi\|_{L^{p'}}
    \end{align}
    where $t\in\mathbb{R}$ and $\ep>0$.
\end{lem}
\begin{proof}
    Let $\tilde{\phi} \in C_{0}^{\infty}(\mathbb{R}^3)$ be satisfying $\tilde{\phi}=1$ on $\mathcal{C}=\{\xi\in \mathbb{R}^3 \mid 3/4 \leq |\xi| \leq 8/3\}$. Then, for any $j\in\mathbb{Z}$, 
    \begin{align*}
        \delj  e^{\pm i\frac{|\nabla|}{\ep}t}\psi = K^{\pm}_j(t,\cdot)\,\star\,\delj \psi,
    \end{align*}
    where 
    \begin{align*}
        K^{\pm}_j(t,x)=2^{3j}\int_{\mathbb{R}^{3}} e^{i 2^jx\cdot \xi}e^{\pm i\frac{|\xi|}{\ep}2^j t} \tilde{\phi}(\xi) d\xi. 
    \end{align*}
    We use the following dispersive estimate 
    \begin{align} \label{j0}
        \|K^{\pm}_j(t,\cdot)\|_{L^\infty} \lesssim 2^{2j} \ep t^{-1}.
    \end{align}
    As for the proof of (\ref{j0}), see \cite[Proposition 8.14]{MR2768550} for example. By using Young's inequality, we have
    \begin{align} \label{inq1}
        \|\delj  e^{\pm i\frac{|\nabla|}{\ep}t}\psi\|_{L^{\infty}} \lesssim 2^{2j} \ep t^{-1} \|\delj \psi\|_{L^1}.
    \end{align}
    Parseval's identity implies
    \begin{align} \label{inq2}
        \|\delj  e^{\pm i\frac{|\nabla|}{\ep}t}\psi\|_{L^{2}} = \|\delj \psi\|_{L^2}.
    \end{align}
    By interpolating the inequalities (\ref{inq1}) and (\ref{inq2}), we obtain the estimate (\ref{j}).
\end{proof}

\section{Low mach number limit of stationary solutions}
This section is devoted to the proof of Theorem \ref{stthm}. We first show the following existence theorem. 
\begin{thm} \label{staexis}
    There exists a constant $c_0>0$ such that  if $\|F\|_{\Bi{-3/2}\cap\dot{H}^4}\leq c_0$, then there exists a unique stationary solution $(\rho,v)=(\rho_{\infty}+\ep b^*_\ep ,v^*_\ep)$ of $(\ref{stationary_eq})$ such that  $b^*_\ep \in \Bi{-1/2}\cap \dot{H}^5$, $v^*_\ep \in \Bi{1/2}\cap \dot{H}^{6}$ and
    \begin{align*}
        \|b_\epsilon^*\|_{\Bi{-\frac{1}{2}}\cap \dot{H}^{5}} + \|v_\epsilon^*\|_{\Bi{\frac{1}{2}}\cap \dot{H}^{6}} \lesssim \|F\|_{\dot{B}^{-\frac{2}{3}}_{2,\infty}\cap\dot{H}^4}.
    \end{align*}
\end{thm} 
{ The overall strategy for proving Theorem \ref{staexis} is based on the approach taken in \cite{MR4803477}. Unlike \cite{MR4803477}, which only shows the existence and uniqueness of (\ref{stationary_eq}) only for the case $\epsilon=1$, our result proves the existence and uniqueness for any $\epsilon > 0$.
To prove Theorem $\ref{staexis}$, we use some lemmas proved in \cite{MR4803477}. The following lemma reformulates the stationry problem (\ref{stationary_eq}). }
\begin{lem} \label{reform} 
    Let $(\rho^*_\ep, v^*_\ep)=(\rho_\infty+\ep b^{*}_{\ep}, v^{*}_\ep)$, $b^*_\ep \in \Bi{-1/2}\cap \dot{H}^5$, $v^*_\ep \in \Bi{{1}/{2}}\cap \dot{H}^{6}$. Then, $(\rho^*_\ep, v^*_\ep)=(\rho_\infty+\ep b^{*}_{\ep}, v^{*}_\ep)$ is the solution to the problem $(\ref{stationary_eq})$ if and only if $(\rho^*_\ep,v^*_\ep)=(\rho_\infty+\ep b^{*}_{\ep}, v^{*}_\ep)$ satisfies the following equations:
    \begin{align} \label{reformeq}
        \left\{ \,
            \begin{aligned}
                &b^*_\ep + \ep^2 \alpha \dive\,(b^{*}_\ep  v^*_\ep) = \ep \gamma^{-2}\Delta^{-1} \dive\,g_\ep , \\
                &v^*_\ep - \ep^{-1}\beta \Delta^{-1}\nabla b^*_\ep = -\mu_1^{-1}\Delta^{-2}\nabla \dive\, g_\ep -\mu^{-1}\Delta^{-1} \mathbb{P}g_\ep, \\
            \end{aligned}
        \right.
    \end{align}
    where $\mu_1 = 2\mu + \mu'$, $\alpha=\mu_1 / (P'(\rho_\infty)\rho_\infty)$, $\beta=P'(\rho_\infty)/\mu_1$, $\gamma=P'(\rho_\infty)^{1/2}$ and 
    \begin{align*}
        &g_\ep(b^*_\ep,v^*_\ep)= -\dive((\rho_\infty +\ep b^*_\ep) v\otimes v) \\
        &\hspace{80pt}-\ep^{-1}(P'(\rho_\infty+\ep b^*_\ep)-P'(\rho_\infty))\nabla  b^*_\ep+ (\rho_\infty+\ep b^*_\ep) F.
    \end{align*}
\end{lem}

{ For $\tilde{v}^*\in L^{\infty}\cap\dot{B}^{5/2}_{2,1}$ and $j\in\mathbb{Z}$, we introduce the perturbed operator $\mathcal{L}_{\tilde{v},j, \ep}:X\to X$ by
\begin{align*}
        &\mathcal{L}_{\tilde{v},j, \ep}(b^*,v^*) =
    \begin{bmatrix}
        b^* + \ep^2 \alpha \dot{S}_{j} \dive\,(b^* \tilde{v}^*) \\
        v^* - \ep^{-1}\beta \Delta^{-1} \nabla b^*
    \end{bmatrix},\\
    &X = \{(b^{*},v^*)\mid b^*\in L^2, v^* \in  \dot{H}^{1} \}.
    \end{align*}
Then, the reformulated equation (\ref{reformeq}) can be written as
\begin{align*}
    \mathcal{L}_{v^{*}_{\ep},j, \ep}(b^*_{\ep},v^*_{\ep}) = N_0(b^*_\ep,v^*_\ep),
\end{align*}
where $N_0$ is given by
\begin{align*}
    N_0(b^{*}_{\ep},v_\ep^*)=
    \begin{bmatrix}
        \ep \gamma^{-2}\Delta^{-1} \dive\,g_\ep \\
        -\mu_1^{-1}\Delta^{-2}\nabla \dive\, g_\ep -\mu^{-1}\Delta^{-1} \mathbb{P}g_\ep
    \end{bmatrix}.
\end{align*}
The following lemma shows the bijectivity of the perturbed operator $\mathcal{L}_{\tilde{v},j, \ep}$
\begin{lem} \label{invert}
    Let $\tilde{v}^{*}\in L^\infty \cap \Bo{5/2}$. If $\|\tilde{v}^*\|_{\Bo{5/2}}$ is small enough, then for any $j\in \mathbb{Z}$ and $\ep>0$, the operator $\mathcal{L}_{\tilde{v},j,\ep} $ 
    is bijective.
\end{lem}
The following lemma provides an estimate for the function $g_\ep$, which will be used in the estimate for $N_0$.
}
\begin{lem} \label{snon}
    Set the function space $Y=Y_0\cap Y_1$ as 
\begin{align*}
    Y_0=\Bi{-\frac{1}{2}}\times \Bi{\frac{1}{2}},\ \ \ Y_1=\dot{H}^{5}\times \dot{H}^{6}.
\end{align*}
    Let $(b^*,v^*), (b^*_1,v^*_1), (b^*_2,v^*_2)\in Y$. Then, there exists a constant $c_1>0$ such that if $F\in \Bi{-3/2}\cap \dot{H}^{3}$ and $\|(b^*,v^*)\|_{Y}, \|(b^*_1,v^*_1)\|_{Y}, \|(b^*_2,v^*_2)\|_{Y} \leq c_1$, then we have 
    \begin{align*}
        \|g_\ep(b^*,v^*)\|_{\Bi{-\frac{3}{2}}\cap \dot{H}^{4}} &\lesssim c_1^2 + \|F\|_{\Bi{-\frac{3}{2}}\cap \dot{H}^{4}}, \\
        \|g_\ep(b^*_1,v^*_1)-g_\ep(b^*_2,v^*_2)\|_{\Bi{-\frac{3}{2}}} &\lesssim (c_1+\|F\|_{\Bi{-\frac{3}{2}}\cap \dot{H}^{4}})  \|(b^*_1-b^*_2,v^*_1-v^*_2)\|_{\Bi{-\frac{1}{2}}\times \Bi{\frac{1}{2}}} .
    \end{align*}
\end{lem}

 The proofs of Lemma $\ref{reform}$, Lemma $\ref{invert}$ and Lemma $\ref{snon}$ are same as in the proofs of \cite[Lemma 3.2, Lemma 3.3 and Lemma 3.4]{MR4803477}.

We show the estimates of the operator $\Phi_{j,\ep} = \mathcal{L}_{\cdot ,j, \ep}^{-1} N_0$ defined on the space $Y=Y_0\cap Y_1$. {  These estimates provide the boundedness of $\Phi_{j,\varepsilon}$ in a neighborhood of the origin in $Y$ and its contractiveness in a neighborhood of the origin in $Y_0$.}

\begin{lem} \label{lip}
    There exists a constant $c_2>0$ such that if $F\in \Bi{-3/2}\cap \dot{H}^{4}$ and $(b^*,v^*), (b^*_1,v^*_1), (b^*_2,v^*_2)\in Y$ satisfy 
    \begin{align*}
        \|F\|_{\Bi{-\frac{3}{2}}\cap \dot{H}^{4}}, \|(b^*,v^*)\|_{Y}, \|(b^*_1,v^*_1)\|_{Y}, \|(b^*_2,v^*_2)\|_{Y} \leq c_2,    
    \end{align*}
    then the operator $\Phi_{j,\ep}(b^*,v^*)=\mathcal{L}_{v^* ,j, \ep}^{-1} N_0(b^*,v^*)$ with $j\geq 0$, $0<\ep\leq 1$ satisfies 
    \begin{align*}
        \|\Phi_{j,\ep}(b^*,v^*)\|_{Y}\lesssim c_2^2 + \|F\|_{\Bi{-\frac{3}{2}}\cap \dot{H}^{4}}
    \end{align*}
    and
    \begin{align*}
        \|\Phi_{j,\ep}(b^*_1,v^*_1)-\Phi_{j,\ep}(b^*_2,v^*_2)\|_{Y_0} \leq c\|(b^*_1-b^*_2,v^*_1-v^*_2)\|_{Y_0},
    \end{align*}
    where $0<c<1$ is a constant independent of $j\geq 0$ and $\ep>0$.
\end{lem}

\begin{proof}
    Let 
    \begin{align*}
        (b,v)^{\mathsf{T}}=\Phi_{j,\ep}(b^*,v^*),\ (b_1,v_1)^{\mathsf{T}}=\Phi_{j,\ep}(b^*_1,v^*_1)\ \mathrm{and}\ (b_2,v_2)^{\mathsf{T}}=\Phi_{j,\ep}(b^*_2,v^*_2).
    \end{align*}
    Then,
    \begin{align*}
        \mathcal{L}_{v^* ,j, \ep}(b,v)=N_0(b^*,v^*),\ \mathcal{L}_{v^*_1 ,j, \ep}(b_1,v_1)=N_0(b^*_1,v^*_1)
    \end{align*}
    and
    \begin{align*}
        \mathcal{L}_{v^*_2 ,j, \ep}(b_2,v_2)=N_0(b^*_2,v^*_2).
    \end{align*}
    We take $c_2>0$ as $c_2\leq c_1$, where $c_1$ is a constant appearing in Lemma \ref{snon}. Then, Lemma \ref{bi1} and Lemma \ref{snon} show that
    \begin{align*}
        \|b\|_{\Bi{-\frac{1}{2}}} &\lesssim \ep^2 \|b v^* \|_{\Bi{\frac{1}{2}}} + \ep \|g_\ep(b^*,v^*)\|_{\Bi{-\frac{3}{2}}}\lesssim \ep c_2^2+ \ep \|F\|_{\Bi{-\frac{3}{2}}\cap \dot{H}^{3}},  \\
        \|v\|_{\Bi{\frac{1}{2}}\cap\dot{H}^{6}} &\lesssim \ep^{-1}\|b\|_{\Bi{-\frac{1}{2}}\cap\dot{H}^{4}} + \|g_\ep(b^*,v^*)\|_{\Bi{-\frac{3}{2}}\cap\dot{H}^{4}} \\
        &\lesssim \ep^{-1}\|b\|_{\dot{H}^{4}} + c_2^2 + \|F\|_{\Bi{-\frac{3}{2}}\cap \dot{H}^{4}}.
    \end{align*}
    
    For any multiindex $\alpha\in \mathbb{Z}^3_{\geq 0}$ with $|\alpha|=5$, we have the estimate
    \begin{align*}
        \|\partial^{\alpha}_x b\|^{2}_{L^2}&= -\ep^2 \alpha \langle \partial_{x}^{\alpha}\dot{S}_{j}\dive(b\,v^*),  \partial_{x}^{\alpha}b \rangle +\ep \gamma^2 \langle \partial_{x}^{\alpha} \Delta^{-1} \dive g_\ep, \partial_{x}^{\alpha} b\rangle \\
        &\lesssim \sum_{\beta\leq\alpha}|\langle \dot{S}_{j}\dive(\partial_{x}^{\beta}b\,\partial_{x}^{\alpha-\beta}v^*),  \partial_{x}^{\alpha}b \rangle| + \|g_\ep\|_{\dot{H}^{3}} \|b\|_{\dot{H}^{5}}.
    \end{align*}
    If $\beta<\alpha$, then we have
    \begin{align*}
        |\langle \dot{S}_{j}\dive(\partial_{x}^{\beta}b\,\partial_{x}^{\alpha-\beta}v^*),  \partial_{x}^{\alpha}b \rangle| \lesssim \|b\|_{H^{5}} \|v\|_{\Bi{\frac{1}{2}}\cap H^6} \|b\|_{\dot{H}^{5}} \lesssim c_2^2 \|b\|_{\dot{H}^{5}}.
    \end{align*}
    By using Lemma \ref{commu} and the identity
    \begin{align*}
        \langle  v^*\cdot \nabla \dot{\Delta}_k^{\frac{1}{2}} \partial_{x}^{\alpha} b,  \dot{\Delta}_k^{\frac{1}{2}} \partial_{x}^{\alpha}b \rangle = -\frac{1}{2} \langle \dive v^* \dot{\Delta}_k^{\frac{1}{2}} \partial_{x}^{\alpha}b, \dot{\Delta}_k^{\frac{1}{2}} \partial_{x}^{\alpha}b  \rangle,
    \end{align*}
    we have
    \begin{align*}
        &|\langle \dot{S}_{j}(v^*\cdot \nabla \partial_{x}^{\alpha}b),  \partial_{x}^{\alpha}b \rangle| \lesssim \sum_{k< j} |\langle \dot{\Delta}_k (v^*\cdot \nabla \partial_{x}^{\alpha}b),  \partial_{x}^{\alpha}b \rangle|\\ 
        &\leq \sum_{k< j} |\langle  v^*\cdot \nabla \dot{\Delta}_k^{\frac{1}{2}} \partial_{x}^{\alpha} b,  \dot{\Delta}_k^{\frac{1}{2}} \partial_{x}^{\alpha}b \rangle| + \sum_{k< j} |\langle  [\dot{\Delta}_k^{\frac{1}{2}}, v^*\cdot \nabla \partial_{x}^{\alpha} b],  \dot{\Delta}_k^{\frac{1}{2}} \partial_{x}^{\alpha}b \rangle|\\
        &\lesssim \|\dive v^*\|_{L^\infty} \sum_{k<j} \|\dot{\Delta}_k^{\frac{1}{2}} \partial_{x}^{\alpha}b\|_{L^2}^2 + \|\nabla v^*\|_{\Bo{\frac{3}{2}}} \|\partial^\alpha_x b\|_{L^2}^2\\
        &\lesssim \|v^*\|_{\Bi{\frac{1}{2}}\cap \dot{H}^6} \|b\|_{\dot{H}^5}^2 \leq c_2 \|b\|_{\dot{H}^5}^2.
    \end{align*}
    By Lemma \ref{bi1},
    \begin{align*}
        |\langle \dot{S}_{j}(\dive v^* \partial_{x}^{\alpha}b),  \partial_{x}^{\alpha}b \rangle| \lesssim  \|v\|_{\Bi{\frac{1}{2}}\cap H^6} \|b\|_{\dot{H}^{5}}^2 \lesssim c_2 \|b\|_{\dot{H}^{5}}^2.
    \end{align*}
    Thus, if $c_2>0$ is small enough, then we obtain
    \begin{align*}
        \|b\|_{\Bi{-\frac{1}{2}}\cap \dot{H}^5} \lesssim c_2^2 + \|F\|_{\Bi{-\frac{3}{2}}\cap \dot{H}^{4}}.
    \end{align*}
    Applying Lemma $\ref{snon}$ for estimating
    \begin{align*}
        v_1-v_2 = \ep^{-1} \beta \Delta^{-1} \nabla (b_1-b_2)+N_2(b^*_1,v^*_1)-N_2(b^*_2,v^*_2),
    \end{align*}
    we obtain
    \begin{align*}
        \|v_1-v_2\|_{\Bi{\frac{1}{2}}} &\lesssim \ep^{-1}\|b_1-b_2\|_{\Bi{-\frac{1}{2}}} + \|g_\ep (b^*_1,v^*_1)-g_\ep (b^*_2,v^*_2)\|_{\Bi{-\frac{3}{2}}}\\
        &\lesssim  \ep^{-1}\|b_1-b_2\|_{\Bi{-\frac{1}{2}}}+c_2\|(b^*_1-b^*_2,v^*_1-v^*_2)\|_{Y_0}.
    \end{align*}
    Let $\Gamma_{l,k}=\dot{\Delta}_l \dot{\Delta}_k^{1/2}$, $l,k\in \mathbb{Z}$ and $\omega = b_1-b_2$. By Lemma \ref{commu}, we have 
    \begin{align*}
        &|\langle \dot{\Delta}_{l} \dot{S}_{j}\dive (b_1 v^*_1-b_2 v^*_2), \deln \omega\rangle| \\
        &\hspace{10pt} \lesssim \sum_{k<j,\ |l-k|\leq 1} \left(\|[\Gamma_{l,k},v^*_{1}\cdot\nabla]\omega\|_{L^2} \|\Gamma_{l,k}\omega\|_{L^2}+ |\langle v^*_1 \cdot \nabla \Gamma_{k,l}\omega, \Gamma_{k,l}\omega\rangle|\right)\\
        &\hspace{50pt} +\left(\|\dot{\Delta}_l (\dive v^*_1\, \omega)\|_{L^2} + \|\dot{\Delta}_l \dive((v^*_1-v^*_2)b_2)\|_{L^2}\right)\|\dot{\Delta}_l \omega\|_{L^{2}}\\
        &\hspace{10pt}  \lesssim 2^{\frac{1}{2}l}\left( \|v^*_1\|_{\Bo{\frac{5}{2}}}\|\omega\|_{\Bi{-\frac{1}{2}}} + \|b_2\|_{\Bo{\frac{3}{2}}}\|v^*_1-v^*_2\|_{\Bi{\frac{1}{2}}} \right)\|\dot{\Delta}_l \omega\|_{L^{2}}\\
        &\hspace{10pt} \lesssim 2^{\frac{1}{2}l}c_2  \|(b^*_1-b^*_2,v^*_1-v^*_2)\|_{Y_0} \|\dot{\Delta}_l \omega\|_{L^{2}}.
    \end{align*}
    For any $l\in \mathbb{Z}$, we have
    \begin{align*}
        &\|\dot{\Delta}_l(b_1-b_2)\|_{L^2}^2=-\ep^2 \langle\dot{\Delta}_l \alpha\dot{S}_{j}\dive (b_1 v^*_1-b_2 v^*_2), \dot{\Delta}_l (b_1-b_2)\rangle\\
        &\hspace{100pt} + \ep \langle\dot{\Delta}_l \gamma^{-2}\Delta^{-1}\dive(g_\ep (b^*_1,v^*_1)-g_\ep (b^*_2,v^*_2)), \dot{\Delta}_l  (b_1-b_2)\rangle\\
        &\lesssim \ep 2^{\frac{1}{2}l}c_2  \|(b^*_1-b^*_2,v^*_1-v^*_2)\|_{Y_0} \|\dot{\Delta}_l (b_1-b_2)\|_{L^{2}}.
    \end{align*}
    If $c_2>0$ is sufficiently small, then we have
    \begin{align*}
        \|b_1-b_2\|_{\Bi{-\frac{1}{2}}} \lesssim \ep c_2  \|(b^*_1-b^*_2,v^*_1-v^*_2)\|_{Y_0}.
    \end{align*}
    Thus, we obtain
    \begin{align*}
        \|(b_1-b_2,v_1-v_2)\|_{Y_0} \lesssim  c_2\|(b^*_1-b^*_2,v^*_1-v^*_2)\|_{Y_0}.
    \end{align*}
\end{proof}
    
\begin{proof}[Proof of Theorem $\ref{staexis}$]
    Let $\|F\|_{\Bi{-3/2}\cap \dot{H}^3}\leq c_2$, where $c_2$ is a constant appearing in Lemma $\ref{lip}$. By Lemma $\ref{lip}$, for any $j\geq 0$ and $\ep>0$, there exists a unique $(b^*_{\ep,j},v^*_{\ep, j})\in Y$ with $\|(b^*_{\ep,j},v^*_{\ep, j})\|_{Y}\lesssim c_2$ such that $(b^*_{\ep,j},v^*_{\ep,j})^{\mathsf{T}}=\Phi_{\ep,j}(b^*_{\ep,j},v^*_{\ep,j})$. Since the sequence $\{(b^*_{\ep,j},v^*_{\ep,j})\}_{j\geq 0}$ is bounded in $Y$, there exists $(b^*_{\ep},v^*_{\ep})\in Y$ with $\|(b^*_{\ep},v^*_{\ep})\|_{Y}\lesssim c_2$ such that there exists some subsequence of $\{(b^*_{\ep,j},v^*_{\ep,j})\}_{j\geq 0}$ converge to $(b^*_{\ep},v^*_{\ep})$ in $\mathcal{S}'$. Then, the limit $(b^*_{\ep},v^*_{\ep})$ is uniquely determined by $F$ since $(b^*_{\ep},v^*_{\ep})$ satisfy the equation ($\ref{reformeq}$). Lemma $\ref{reform}$ implies that $(\rho^{*}_{\ep}, v^*_\ep)=(\rho_\infty+ \ep b^*_{\ep},v^*_{\ep})$ is a solution of ($\ref{stationary_eq}$).
\end{proof}
The following existence result for the stationary problem (\ref{istationary_eq}) is the special case of the theorem in \cite[Theorem 1.1]{MR3978259}.
\begin{thm}[\cite{MR3978259}] \label{kks}
    There exists a constant $c_3>0$ such that if $F\in \Bi{-3/2}$ satisfies
    \begin{align*}
        \|F\|_{\Bi{-\frac{3}{2}}} \leq c_3,
    \end{align*}
    then there exists a unique solution $u^*\in \Bi{1/2}$ of $(\ref{istationary_eq})$ such that 
    \begin{align*}
        \|u^*\|_{\Bi{\frac{1}{2}}} \lesssim \|F\|_{\Bi{-\frac{3}{2}}}.
    \end{align*}
\end{thm}
The rest of this section is devoted to proving Theorem \ref{stthm}. 
\begin{proof}[Proof of Theorem $\ref{stthm}$]
    Let $u^*$ be a solution of ($\ref{istationary_eq}$) obtained in Theorem $\ref{kks}$ and let $(\rho^*_\ep, v^*_\ep)=(\rho_\infty+\ep b^{*}_{\ep}, v^{*}_\ep)$ be a solutions of  (\ref{stationary_eq}) obtained in Theorem $\ref{staexis}$. Then, $u^*$ and $(b^{*}_{\ep}, v^{*}_\ep)$ satisfy
    \begin{align*}
        \|u^*\|_{\Bi{\frac{1}{2}}} + \|b_\epsilon^*\|_{\Bi{-\frac{1}{2}}\cap \dot{H}^{4}} + \|v_\epsilon^*\|_{\Bi{\frac{1}{2}}\cap \dot{H}^{5}} \lesssim \|F\|_{\dot{B}^{-\frac{3}{2}}_{2,\infty}\cap\dot{H}^3}.
    \end{align*}
    Since $(\rho^*_\ep, v^*_\ep)$ satisfies the equation ($\ref{istationary_eq}$), we have
    \begin{equation} \label{stationary_eq2}
        \left\{ \,
        \begin{aligned}
            &{\dive v^*_\ep} = -\frac{\ep}{\rho_\infty}\dive (b^*_\ep v^*_\ep), \\
            &\mu \Delta v^*_\epsilon + (\mu+\mu')\nabla \dive\,v^*_\epsilon -P'(\rho_\infty) \frac{\nabla b^*}{\ep} = -g_\ep(b^*,v^*),\\
        \end{aligned}
        \right.
    \end{equation}
    where $g_\ep(b^*,v^*)$ is defined in Lemma $\ref{reform}$.
    Let $v^1_\ep = \mathbb{P}v^*_\ep-u^*$, $v_2= \mathbb{Q}v^*_\ep$. Then, we have
    \begin{align*}
        &v^1_\ep = \mu^{-1} \mathbb{P}\Delta^{-1}(\ep \dive (b_\ep^* v^*_\ep \otimes v^*_\ep)+ \dive(v^1_\ep \otimes \vep)+\dive(u^*\otimes v^1_\ep) + \ep b^*_\ep F), \\
        &v^2_\ep = - \ep \Delta^{-1} \nabla \dive (b^*_\ep v^*_\ep),\\
        &\nabla b^*_\ep = - (P'(1+\ep b^*_\ep )-P'(1)) \nabla b_\ep + \ep (\mu \Delta v^*_\epsilon + (\mu+\mu')\nabla \dive\,v^*_\epsilon) \\
        &\hspace{150pt} - \ep \dive ((1+\ep b^*_\ep )v^*_\ep\otimes v^*_\ep) + \ep (1+\ep b^*_\ep)F,
    \end{align*}
    By Lemma \ref{bi1},  
    \begin{align*}
        &\|v^1_\ep\|_{\Bi{\frac{1}{2}}} \lesssim \ep \|b^*_\ep\|_{\Bo{\frac{3}{2}}} \|\vep\|_{\Bi{\frac{1}{2}}}^2 + (\|\vep\|_{\Bi{\frac{1}{2}}} + \|u^*\|_{\Bi{\frac{1}{2}}})\|v^1_\ep\|_{\Bi{\frac{1}{2}}} \\
        &\hspace{250pt}+\ep \|\bep\|_{\Bo{\frac{3}{2}}} \|F\|_{\Bi{-\frac{3}{2}}},\\
        &\|v^2_\ep\|_{\Bi{\frac{1}{2}}} \lesssim  \ep \|\bep\|_{\Bo{\frac{3}{2}}} \|\vep\|_{\Bi{\frac{1}{2}}},\\
        &\| b^*_\ep\|_{\Bi{-\frac{1}{2}}} \lesssim \ep \|\bep\|_{\Bo{\frac{3}{2}}} \|\bep\|_{\Bi{-\frac{1}{2}}} + \ep(1+\ep\|\bep\|_{\Bo{-\frac{3}{2}}}) \|\vep\|_{\Bi{\frac{1}{2}}}.
    \end{align*}
    Thus, we obtain
    \begin{align*}
        \|b^*_\epsilon\|_{\dot{B}^{-\frac{1}{2}}_{2,\infty}} + \|(v^1_\ep, v^2_\ep)\|_{\dot{B}^{\frac{1}{2}}_{2,\infty}} \lesssim \epsilon\ \ \ \mathrm{for}\ \ \ \epsilon \ll 1.
    \end{align*}
\end{proof}
\section{Low mach number limit of non-stationary solutions}
This section is devoted to proving Theorem \ref{nstthm}. Let $(\rho^*_\ep,v^*_\ep)=(\rho_\infty+\ep b^*_\ep,v^*_\ep)$ be a solution of $(\ref{stationary_eq})$ and let $u^*$ be a solution of $(\ref{istationary_eq})$, both of which satisfy the estimate
\begin{align*}
     \|b_\epsilon^*\|_{\Bi{-\frac{1}{2}}\cap \dot{H}^{4}} + \|v_\epsilon^*\|_{\Bi{\frac{1}{2}}\cap \dot{H}^{5}} + \|u^*\|_{\Bi{\frac{1}{2}}} \lesssim \|F\|_{\dot{B}^{-\frac{2}{3}}_{2,\infty}\cap\dot{H}^3}.
\end{align*}

Let $\rho_\ep = \rho_\infty +\ep b_\ep$. The perturbation $(\sigma_\ep,w_\ep)=(b_\ep-b_\ep^*,v_\ep-v^*_\ep)$ satisfies the following system of equations:
\begin{align} \label{peq}
    \left\{ \,
    \begin{aligned}
        &\partial_{t}\sigma_\ep +  \rho_\infty \frac{\dive\,w_\ep}{\ep} = f_\ep(\sigma_\ep,\sigma_\ep^*, w_\ep, v_\ep^*), \\
        &\partial_{t}w_\ep - \mathcal{A} w_\ep + \gamma_0 \frac{\nabla\sigma_\ep}{\ep} = g_\ep(\sigma_\ep,\sigma_\ep^*, w_\ep, v_\ep^*), \\
        &(\sigma_\ep,w_\ep)|_{t=0} = (\sigma_{\ep,0}, w_{\ep,0}),
    \end{aligned}
    \right.
\end{align}
where $\gamma_0=P'(\rho_\infty)/\rho_\infty$, $\mathcal{A}= \mu \Delta + (\mu+\mu')\nabla \dive$, $(\sigma_{\ep,0},w_{\ep,0})=(b_{\ep,0} - b^*_{\ep}, v_{\ep,0} - v^*_\ep)$; $f$ and $g$ are defined by the following:
\begin{align*}
    f_\ep (\sigma_\ep,\sigma_\ep^*, w_\ep, v_\ep^*) = -\dive \left\{ (v^{*}_\ep+w_\ep)\sigma_\ep+\sigma^*_\ep w_\ep \right\},\ 
    g_\ep (\sigma_\ep,\sigma_\ep^*, w_\ep, v_\ep^*) = \sum_{i=1}^{4}g^{i}
\end{align*}
with
\begin{align*}
    g^1_\ep &= -v^*_\ep \cdot \nabla w_\ep - w_\ep \cdot \nabla v^*_\ep - w_\ep\cdot \nabla w_\ep,\\
    g^{2}_\ep &= -\ep^{-1}(\Phi(\ep\sigma^*_\ep + \ep\sigma_\ep)-\Phi(\ep\sigma^{*}_\ep)) \nabla \sigma^*_\ep - \ep^{-1}(\Phi(\ep\sigma^*_\ep + \ep\sigma_\ep)-\Phi(0))\nabla \sigma_\ep, \\
    g^{3}_\ep &= (\Psi(\ep\sigma^* + \ep\sigma)-\Psi(\ep\sigma^*))\mathcal{A} (v^* + w), \ \ g^{4}_\ep = (\Psi(\ep\sigma^*)-\Psi(0)) \mathcal{A} w,\\
    \Phi(&\zeta)=\frac{P'(\rho_\infty +\zeta)}{\rho_\infty +\zeta},\ \ \Psi(\zeta)=\frac{1}{\rho_\infty+\zeta}.
\end{align*}

{ The following theorem is an extension of the existence theorem for the non-stationary problem obtained in the $\dot{B}^{\frac{1}{2}}_{2,\infty}\cap\dot{H}^3$ framework in \cite{MR4803477} to the $\dot{B}^{\frac{1}{2}}_{2,\infty}\cap\dot{H}^4$ framework. The estimate in $\dot{H}^4$ norm will be used in the estimate of the term $g_\ep^4$ in the proof of Theorem \ref{nstthm} below. }
\begin{thm} \label{time-decay}
    There exists a constant $c_4>0$ such that if $F\in \Bi{-3/2}\cap \dot{H}^4$ and $(\sigma_{\ep,0},w_{\ep,0})\in \Bi{1/2}\cap \dot{H}^4$ satisfy
    \begin{align} \label{smallassumption3}
        \|(\sigma_{\ep,0},w_{\ep,0})\|_{\Bi{\frac{1}{2}}\cap \dot{H}^4} + \|F\|_{\Bi{-\frac{3}{2}}\cap \dot{H}^4} \leq c_4,
    \end{align}
    then there exists a unique solution $(\sigma_\ep,v_\ep)$ of $(\ref{peq})$ satisfying 
    \begin{align*}
        (\sigma_\ep,v_\ep)\in C^0([0,\infty);\Bi{\frac{1}{2}}\cap\dot{H}^4),
    \end{align*}
    \begin{align} \label{inifor}
        \sup_{0\leq t<\infty} \|(\sigma_\ep,v_\ep)(t)\|_{\Bi{\frac{1}{2}}\cap \dot{H}^4}\lesssim c_1
    \end{align}
    and 
    \begin{align} \label{wekest}
        \|(\sigma_\ep,v_\ep)(t)\|_{\Bi{s}\cap \dot{H}^4} \lesssim_{s} (1+t)^{-\frac{s}{2}+\frac{1}{4}} \|(\sigma_{\ep,0},v_{\ep,0})\|_{\Bi{\frac{1}{2}}\cap\dot{H}^4}
    \end{align}
    holds for $-3/2<s<3/2$ with $s_0\leq s$ and $t\geq 0$.
\end{thm}
\begin{proof}
    Since the existence in $\Bi{1/2}\cap \dot{H}^3$ framework and time-decay estimate
    \begin{align*}
        \|(\sigma_\ep,v_\ep)(t)\|_{\Bi{s}} \lesssim_{s} (1+t)^{-\frac{s}{2}+\frac{1}{4}} \|(\sigma_{\ep,0},v_{\ep,0})\|_{\Bi{\frac{1}{2}}\cap\dot{H}^3}
    \end{align*}
    for $-3/2<s<3/2$ with $s\leq s_0$ are proved in \cite[Theorem 1.2]{MR4803477}, we only show the estimate
    \begin{align*}
        \|(\sigma_\ep,w_\ep)(t)\|_{\dot{H}^4} \lesssim_{s} (1+t)^{-\frac{s}{2}+\frac{1}{4}} \|(\sigma_{\ep,0},v_{\ep,0})\|_{\Bi{\frac{1}{2}}\cap\dot{H}^4},
    \end{align*}
    where $-3/2<s<3/2$ with $s_0\leq s$ and $t\geq 0$. Fix $j_0\in\mathbb{Z}$ and let $(\sigma_{\ep,H},w_{\ep,H})=(1-\dot{S}_{j_0})(\sigma_{\ep},w_{\ep})$ and $(f_{\ep,H},g_{\ep,H})=(1-\dot{S}_{j_0})(f_{\ep},g_{\ep})$. Since $(\sigma_\ep,v_\ep)$ is a solution of ($\ref{peq}$), for any multi-index $\alpha_1, \alpha_2 \in\mathbb{Z}^3$ with $|\alpha_1|=4$, $|\alpha_2|=3$, we have
    \begin{align*}
        &\frac{1}{2}\frac{d}{dt}\|\partial_{x}^{\alpha_1}(\gamma_1 \sigma_{\ep,H},w_{\ep,H})\|_{L^2}^{2} + \mu \|\partial_{x}^{\alpha_1}\nabla w_{\ep,H}\|_{L^2}^2 + (\mu+\mu')\|\partial_{x}^{\alpha_1}\dive\,w_{\ep,H}\|_{L^2}^2 \\
        &\hspace{150pt}=  \gamma_1\langle \partial_{x}^{\alpha_1} f_{\ep,H}, \partial_{x}^{\alpha_1} \sigma_{\ep,H}\rangle + \langle \partial_{x}^{\alpha_1} g_{\ep,H}, \partial_{x}^{\alpha_1} w_{\ep,H} \rangle, \\
        &\ep \frac{d}{dt}\langle \partial_{x}^{\alpha_2}\nabla \sigma_{\ep,H}, \partial_{x}^{\alpha_2} w_{\ep,H} \rangle + \gamma_0\|\partial_{x}^{\alpha_2}\nabla\sigma_{\ep,H}\|_{L^2}^2 \\
        &\hspace{50pt}=\rho_\infty\|\partial_{x}^{\alpha_2}\dive\, w_{\ep,H}\|_{L^2}^2+ \ep \langle\partial_{x}^{\alpha_2}\mathcal{A} w_{\ep,H},\partial_{x}^{\alpha_2}\nabla \sigma_{\ep,H}\ \rangle\\
        &\hspace{140pt}+ \ep\langle \partial_{x}^{\alpha_2}\nabla f_{\ep,H},\partial_{x}^{\alpha_2} w_{\ep,H} \rangle + \ep\langle \partial_{x}^{\alpha_2}g_{\ep,H}, \partial_{x}^{\alpha_2}\nabla \sigma_{\ep,H} \rangle,
    \end{align*}
    where $\gamma_1= P'(\rho_\infty)/\rho_\infty^2$.
    By using Lemma \ref{bi1} and the identity
    \begin{align*}
        \langle v_{\ep}\cdot \nabla \partial_{x}^{\alpha_1}\sigma_{\ep,H},\partial_{x}^{\alpha_1}\sigma_{\ep,H} \rangle = -\frac{1}{2} \langle \dive v_\ep\, \partial_{x}^{\alpha_1}\sigma_{\ep,H},\partial_{x}^{\alpha_1}\sigma_{\ep,H} \rangle,
    \end{align*}
    we have
    \begin{align*} 
        &\langle\partial_{x}^{\alpha_1} f_{\ep,H}, \partial_{x}^{\alpha_1} \sigma_{\ep,H}\rangle =  \langle (v_\ep \cdot\nabla\partial_{x}^{\alpha_1} \sigma_\ep)_H, \partial_{x}^{\alpha_1} \sigma_{\ep,H} \rangle +\langle ((\dive\,v_\ep)\partial_{x}^{\alpha_1} \sigma_\ep)_H, \partial_{x}^{\alpha_1} \sigma_{\ep,H} \rangle \nonumber \\
        &\hspace{50pt}+\sum_{0< \beta\leq \alpha_1} \langle \dive(\partial^{\beta}_{x}v_\ep\, \partial_{x}^{\alpha_1-\beta} \sigma_\ep )_H, \partial_{x}^{\alpha_1} \sigma_{\ep,H} \rangle  + \langle \partial_x^{\alpha_1}\dive(\sigma^*_\ep w_\ep)_H, \partial_{x}^{\alpha_1} \sigma_{\ep,H} \rangle \nonumber \\
        &\hspace{5pt}\lesssim \|\dive v_\ep\|_{L^\infty} \|\partial_{x}^{\alpha_1}\sigma_{\ep,H}\|_{L^{2}}^2 + \|v_\ep\|_{\Bo{\frac{3}{2}}\cap\Bo{\frac{5}{2}}}\| \partial_x^{\alpha_1}\sigma_\ep\|_{L^2}\|\partial_x^{\alpha_1}\sigma_{\ep,H}\|_{L^2}\nonumber \\
        &\hspace{60pt} + \| \nabla v_\ep\|_{H^{4}} \|\nabla\sigma_\ep\|_{H^{3}}\|\partial_x^{\alpha_1}\sigma_{\ep,H}\|_{L^{2}} + \| \nabla \sigma^*_\ep\|_{H^{4}} \|\nabla w_\ep\|_{H^{4}}\|\partial_x^{\alpha_1}\sigma_{\ep,H}\|_{L^{2}} \nonumber \\
        &\hspace{5pt}\lesssim  c_4 \big( \|(\sigma_\ep,w_\ep)\|_{\Bi{s}\cap\dot{H}^{4}} + \|w_{\ep,H}\|_{\dot{H}^{5}} \big) \|\partial_x^{\alpha_1}\sigma_H\|_{L^{2}}.
    \end{align*}
    By Lemma \ref{bi1} and Lemma \ref{compos}, we have
    \begin{align*}
        &\langle \partial_{x}^{\alpha_1} g_{\ep,H}, \partial_{x}^{\alpha_1} w_{\ep,H} \rangle  \\
        &\hspace{50pt}\lesssim c_4 \|(\sigma_\ep,v_\ep)\|_{\Bi{s}\cap\dot{H}^{4}}^{2}+c_4 (\|\partial_x^{\alpha_1} w_{\ep,H}\|_{L^2}^2+\|\partial_x^{\alpha_2}\nabla \sigma_{\ep,H}\|_{L^2}^2)
    \end{align*}
    and 
    \begin{align*}
        &\sum_{|\alpha_2|=2} \langle \partial_{x}^{\alpha_2}g_{\ep,H}, \partial_{x}^{\alpha_2}\nabla \sigma_{\ep,H} \rangle \\
        &\hspace{50pt}\lesssim c_4 \|(\sigma_\ep,v_\ep)\|_{\Bi{s}\cap\dot{H}^{4}}^{2}+c_4 (\|\partial_x^{\alpha_1} w_{\ep,H}\|_{L^2}^2+\|\partial_x^{\alpha_2}\nabla \sigma_{\ep,H}\|_{L^2}^2).
    \end{align*}
    Since $\|\partial_x^{\alpha_1} w_{\ep,H}\|_{L^2}\lesssim_{j_0}\|\nabla\partial_x^{\alpha_1} w_{\ep,H}\|_{L^2}$, if $\eta>0$ is small enough, then we have the estimate
    \begin{align*}
        \frac{d}{dt}\mathcal{E}_{\eta}(t) + c  \mathcal{E}_\eta (t) \lesssim_{j_0} c_4 \|(\sigma_\ep,v_\ep)\|_{\Bi{s}\cap\dot{H}^{4}}^{2}\ \ \ \mathrm{for}\ \ \ 0< t<T,
    \end{align*}
    where $c>0$ is a constant and 
    \begin{align*}
        \mathcal{E}_{\eta}(t)= \sum_{|\alpha_1|=4}\|\partial_{x}^{\alpha_1}(\gamma_1 \sigma_{\ep,H},w_{\ep,H})(t)\|_{L^2}^{2} + \ep \sum_{|\alpha_2|=3}\eta \langle \partial_{x}^{\alpha_2}\nabla \sigma_{\ep,H}(t), \partial_{x}^{\alpha_2} w_{\ep,H}(t) \rangle.
    \end{align*}
    Since $\mathcal{E}_{\eta}\sim \|(\sigma_{\ep,H},w_{\ep,H})\|_{\dot{H}^4}$ if $\eta>0$ is small, by Gr\"{o}nwall's inequality, we have
    \begin{align*}
        \|(\sigma_{\ep,H},w_{\ep,H})(t)\|_{\dot{H}^{4}}^2 &\lesssim  e^{-c t}\|(\sigma_{\ep,H},w_{\ep,H})(0)\|_{\dot{H}^4}^2\\
        &\hspace{50pt}+c_4 \int_{0}^{t} e^{-c_0(t-\tau)}\|(\sigma_{\ep},w_{\ep})(\tau)\|_{\Bi{s}\cap\dot{H}^{4}}^{2}d\tau\\
        &\lesssim (1+t)^{-s+\frac{1}{2}}(\|(\sigma_{\ep,0},w_{\ep,0})\|_{\dot{H}^4}^2 + c_4 \mathcal{D}_{s}^2),
    \end{align*}
    where $t\geq 0$ and
    \begin{align*}
        \mathcal{D}_s = \sup_{t\geq 0} (1+t)^{\frac{s}{2}-\frac{1}{4}}\|(\sigma_\ep,w_\ep)(t)\|_{\Bi{s}\cap \dot{H}^4}.
    \end{align*}
    Thus, we obtain $\mathcal{D}_s\lesssim \|(\sigma_{\ep,0},w_{\ep,0})\|_{\Bi{\frac{1}{2}}\cap\dot{H}^{4}}$.
\end{proof}
The following existence result of the incompressible Navier-Stokes equation ($\ref{ieq}$) is the special case of the theorem in \cite[Theorem 1.1]{MR4465911}. (Cf. \cite{MR4608839}, \cite{MR1274547} and \cite{MR1777114}.)  
\begin{thm}[\cite{MR4465911}] \label{iexis}
    There exists a constant $c_5>0$ such that if 
    \begin{align*}
        \|\mathbb{P}v_0\|_{\Bi{\frac{1}{2}}} + \|F\|_{\Bi{-\frac{3}{2}}} \leq c_5,
    \end{align*}
    then there exists a unique global solution $u\in C^{0}([0,\infty);\Bi{{1}/{2}})$ of $(\ref{ieq})$ satisfies
    \begin{align*}
        \sup_{t>0} \|u(t)\|_{\Bi{\frac{1}{2}}} \lesssim c_5.
    \end{align*}
\end{thm}
{  The solution $u$ of $(\ref{ieq})$ obtained by Theorem \ref{iexis} is the target of convergence for $v_\epsilon$ as $\ep\to 0$.

In order to prove the low Mach number limit for the perturbation $(\sigma_\ep,w_\ep)$, we use the Strichartz type estimate for the semigroup $e^{tA_\ep}$ defined by
\begin{align} \label{semi}
    e^{tA_\ep}U_0 = \mathcal{F}^{-1} \left[ e^{t\hat{A_\ep}(\xi)} \widehat{U_0} \right],\ \ \ U_0=(U_{0,1},\ldots,U_{0,4})^{\mathsf{T}}\in \mathcal{S}'(\mathbb{R}^3)^{4},
\end{align}
where $\hat{A_\ep}(\xi)$ is the matrix of the form:
\begin{align*}
    \hat{A_\ep}(\xi) = 
    \begin{bmatrix}
        0 & -i\ep^{-1}\xi^{\mathsf{T}}\\
        -i\ep^{-1} \xi & -2\mu_0 \xi \otimes \xi
    \end{bmatrix}.
\end{align*}

The following states the Strichartz type estimate for the semigroup $e^{tA}$.
\begin{prop} \label{strichartztypeestimate}
    \begin{enumerate}[\normalfont{(}i\normalfont{)}]
        \item Let $2\leq p<\infty$, $2< r\leq \infty$ and $s, s_1, s_2 \in \mathbb{R}$ with $s_1+2/r<s<s_2$. Then, for any $\psi\in \Bi{s_1}\cap \Bi{s_2+3(1/2-1/p)}$, we have
    \begin{align*}
        \|e^{tA_\ep} \psi \|_{L^{r}_t (\dot{B}^{s}_{p,1})} \lesssim_{s,s_1, s_2, p,r} \max\left(\ep^{\frac{1}{r
    }}, \ep^{\frac{1}{2}-\frac{1}{p}}\right)\|\psi\|_{\Bi{s_1}\cap \Bi{s_2+{3}\left(\frac{1}{2}-\frac{1}{p}\right)}
    }.
    \end{align*}
        \item Let $2\leq p <\infty$, $2<r\leq \infty$ and $s\in \mathbb{R}$. Then, for any $\Psi \in L^{r}_t (\dot{B}_{p',1}^{s+2-8/p}\cap \dot{B}^{s+3/2-3/p}_{2,1})$, we have 
    \begin{align*}
        \left\|\int_0^t e^{\tau A_\ep}\Psi(t-\tau)d\tau \right\|_{L^r_t(\dot{B}_{p,1}^{s})}  \lesssim_{p,r} \ep^{1-\frac{2}{p}} \|\Psi\|_{L^r_t(\dot{B}^{s+2-\frac{8}{p}}_{p',1}\cap \dot{B}^{s+\frac{3}{2}-\frac{3}{p}}_{2,1})}.
    \end{align*}
    \end{enumerate}
\end{prop}
}
The eigenvalues of $\hat{A_\ep}(\xi)$ are given by
\begin{align} \label{eigen}
    \lambda_{\pm}(\xi) = -\mu_0|\xi|^{2} \pm \sqrt{\mu_0^{2}|\xi|^{4}-\ep^{-2} |\xi|^{2}},\ \ \lambda_{0}(\xi)=0, 
\end{align}
where $\mu_0=\mu+\mu'/2$.
We set $P_{\pm,\ep}(\xi)$:
\begin{align} \label{projm}
    P_{\pm}(\xi) = \frac{E_{\pm}\otimes E_{\pm}}{E_{\pm}\cdot E_{\pm}}\ \ \mathrm{with}\ \ E_{\pm}=\begin{bmatrix}
        -i\lambda_{\pm}^{-1}\ep^{-1} |\xi|^2\\
         \xi
    \end{bmatrix}.
\end{align} 
We have the spectral resolution
\begin{align} \label{spe1}
    e^{t\hat{A_\ep}(\xi)} = e^{\lambda_{+}t}P_{+}(\xi) + e^{\lambda_{-}t}P_{-}(\xi) \ \ \ \mathrm{for}\ \ \ |\xi|\neq 0,\ep^{-1}\mu_0^{-1},
\end{align}
and if $|\xi|=\ep^{-1} \mu_0^{-1}$, then 
\begin{align} \label{eta0}
    e^{t\hat{A}_\ep (\xi)} = e^{-\mu_0 |\xi|^2 t} \begin{bmatrix}
        1-\mu_0 |\xi|^2 t &-i \ep^{-1} \xi^{\mathsf{T}} t  \\
        -i \ep^{-1} \xi t  &(1-\mu_0 |\xi|^2 t) \frac{\xi \otimes \xi}{|\xi|^2}
    \end{bmatrix}.
\end{align}
For any $V\in \mathbb{C}^4$ and $\xi\in\mathbb{R}^3$ with $|\xi|\neq 0, \ep^{-1}\mu_0^{-1}$, the following property holds: There exists a constant $\ep_0>0$ such that if $\ep\leq \ep_0$, then
\begin{align} \label{upper}
    |P_{+}(\xi)V|^2 + |P_{-}(\xi)V|^2 \lesssim_{\ep_0} |V|^2. 
\end{align}

{ The following lemma provides the $L^{p}$ ($1<p<\infty$) boundedness of the low-frequency part of the Fourier multipliers $\delj \mathcal{F}^{-1}[P_{\pm}\hat{\psi}]$.}
\begin{lem} \label{multi}
    Let $1< p < \infty$. There exists a constant $d_0>0$ such that, for any $\psi\in L^p$, $0<\ep\leq 1$ and $\ep 2^j\leq d_0$ with $j\in\mathbb{Z}$, we have
    \begin{align*}
        \|\delj \mathcal{F}^{-1}[P_{\pm}\hat{\psi}]\|_{L^p} \lesssim \|\delj \psi\|_{L^p}.
    \end{align*}
\end{lem}
\begin{proof}
    We rewrite $P_{\pm}$ as
    \begin{align}
        P_{\pm}(\xi) = \frac{1}{(k_{\pm}(\ep |\xi|))^{2}+\ep^2} \begin{bmatrix}
        (k_{\pm}(\ep |\xi|))^2 & -i\ep k_{\pm}(\ep |\xi|) \frac{\xi^{\mathsf{T}}}{|\xi|}\\
        -i\ep k_{\pm}(\ep |\xi|) \frac{\xi}{|\xi|} &  \ep^2\frac{\xi \otimes \xi}{|\xi|^2}
    \end{bmatrix},
    \end{align}
    where
    \begin{align*}
        k_{\pm}(y):=\mu_0 y \pm i\sqrt{1-\mu_0 y^2},\ \ \ y\in\mathbb{R}. 
    \end{align*}
    There exists a constant $\tilde{d_0}>0$ such that 
    \begin{align*}
        |\partial_\xi^{\alpha}k_{\pm}(\ep|\xi|)| \lesssim_{\alpha} \sum_{|\beta| \leq|\alpha|} \sum_{l=0}^{|\alpha|} |\partial_{\xi}^{\beta}(\ep|\xi|)| |k^{(l)}_{\pm}(\ep|\xi|)|\lesssim_{\tilde{d_0}} |\xi|^{-|\alpha|} 
    \end{align*}
    for any multi-index $\alpha$ and $\xi\in \mathbb{R}^3$ with $\ep|\xi|\leq \tilde{d_0}$.
    If $\tilde{d_0}>0$ is small enough, then for any multi-index $\alpha$, we have the estimate
    \begin{align*}
        |\partial^{\alpha}_{\xi} P_{\pm}(\xi)| \lesssim_{\alpha, \tilde{d_0}} |\xi|^{-|\alpha|},\ \ \ \ep|\xi|\leq \tilde{d_0}.  
    \end{align*}
    Thus, by Mikhlin's multiplier theorem (see \cite[3.2 Theorem 3]{MR290095} for example), there exists $d_0>0$ such that
    \begin{align*}
        \|\delj \mathcal{F}^{-1}[P_{\pm}\hat{\psi}]\|_{L^p} \lesssim \|\delj \psi\|_{L^p},
    \end{align*}
    where $\psi\in L^p$, $\ep 2^j\leq d_0$.
\end{proof}
{  The following lemma extracts the effect of the heat semigroup from the semigroup $e^{tA_\epsilon}$. }
\begin{lem} \label{approeA}
    Let $1< p < \infty$ and $j\in\mathbb{Z}$. Then, there exists a constant $d_0>0$ such that if $\ep 2^j \leq d_1$, then
    \begin{align*}
        \|\delj e^{tA_\ep}\psi \|_{L^p} \lesssim e^{-c 2^{2j}t} \|\delj e^{ t\frac{|\nabla|}{\ep}} \psi\|_{L^p}, 
    \end{align*}
    where $c>0$ is a constant. 
\end{lem}
\begin{proof}
    By the spectral resolution (\ref{spe1}), for any $j\in\mathbb{Z}$, 
    \begin{align*}
        \|\delj e^{tA_\ep}\psi\|_{L^{p}} \leq \sum_{\pm} \|\delj \mathcal{F}^{-1}[\eta_{\pm,\ep}e^{\lambda_{\pm,0}t}P_{\pm} \hat{\psi}]\|_{L^p},
    \end{align*}
    where $\eta_{\pm}=e^{(\lambda_\pm -\lambda_{\pm,0})t}$, $\lambda_{\pm,0}=-\mu_0 |\xi|^2 \pm i|\xi|/\ep$. By using Young's inequality and the fact that $\delj = \sum_{l=-1}^{1} \delj \dot{\Delta}_{j+l}$, we have
    \begin{align} \label{1}
        &\|\delj \mathcal{F}^{-1}[\eta_{\pm}e^{\lambda_{\pm,0}t}P_{\pm} \hat{\psi}]\|_{L^p}\lesssim \sum_{l=-1}^{1} \|\delj \dot{\Delta}_{j+l} \mathcal{F}^{-1}[\eta_{\pm,\ep}e^{\lambda_{\pm,0}t}P_{\pm} \hat{\psi}]\|_{L^p}\nonumber \\ 
        &\hspace{50pt}\lesssim \sum_{l=-1}^{1}\|\mathcal{F}^{-1}[\eta_{\pm} \phi^2(2^{-(j+l)}\cdot)\|_{L^1} \|\delj \mathcal{F}^{-1}[e^{\lambda_{\pm,0}t}P_{\pm} \hat{\psi}]\|_{L^p},
    \end{align} 
    where $\phi^2$ is the symbol of the Fourier multiplier $\dot{\Delta}_0$.
    Define the function $\tilde{k}$ by
    \begin{align*}
        \tilde{k}(y):= \frac{\mu_0^2}{1+\sqrt{1-\mu_0^2 y^2}},\ \ \ y\in\mathbb{R}.
    \end{align*}
    By estimating the power series
    \begin{align*}
        \eta_{\pm}(\xi)=\sum_{n=0}^{\infty} \frac{(\mp i\ep t |\xi|^3 \tilde{k}(\ep |\xi|))^n}{n!},
    \end{align*}
    we have
    \begin{align*}
        &\sum_{l=-1}^{1}\|\mathcal{F}^{-1}[\eta_{\pm} \phi^2(2^{-(j+l)}\cdot)]\|_{L^1} \\
        &\hspace{50pt}\lesssim \sum_{n=0}^{\infty} \frac{(\ep t)^n}{n!} \sum_{l=-1}^{1}\|\mathcal{F}^{-1}[( |\cdot|^3 \tilde{k}(\ep |\cdot|))^n \phi^2(2^{-(j+l)}\cdot)]\|_{L^1}\\
        &\hspace{50pt} \lesssim \sum_{n=0}^{\infty} \frac{(\ep 2^{3j}t)^n}{n!} \sum_{l=-1}^{1}\|\mathcal{F}^{-1}[( |\cdot|^3 \tilde{k}(\ep 2^{j+l}|\cdot|))^n \phi^2(\cdot)]\|_{L^1}.
    \end{align*}
    Let $\tilde{\phi}(\xi):=\sum_{l=-2}^{2}\phi^2(2^{-l}\xi), \xi\in\mathbb{R}^3$. Then $\phi^2 = \tilde{\phi}^{m} \phi^2$ for any $m\geq 1$. By using Young's inequality, for any $n\geq 1$, we have
    \begin{align*}
        &\sum_{l=-1}^{1}\|\mathcal{F}^{-1}[( |\cdot|^3 \tilde{k}(\ep 2^{j+l}|\cdot|))^n \phi^2(\cdot)]\|_{L^1} \\
        &\hspace{50pt}= \sum_{l=-1}^{1}\|\mathcal{F}^{-1}[( |\cdot|^3 \tilde{k}(\ep 2^{j+l}|\cdot|))^n \tilde{\phi}^{2n}(\cdot)\phi^2(\cdot)]\|_{L^1}\\
        &\hspace{50pt}\lesssim \sum_{l=-1}^{1} (\|\mathcal{F}^{-1}\phi^2\|_{L^1} \|\mathcal{F}^{-1}[|\cdot|^3\tilde{\phi}(\cdot) \tilde{k}(\ep 2^{j+l}|\cdot|) \tilde{\phi}(\cdot)]\|_{L^1})^n\\
        &\hspace{50pt}\lesssim \sum_{l=-1}^{1} (C\|\mathcal{F}^{-1}[|\cdot|^3\tilde{\phi}(\cdot)]\|_{L^1} \|\mathcal{F}^{-1}[\tilde{k}(\ep 2^{j+l}|\cdot|) \tilde{\phi}(\cdot)]\|_{L^1})^n,
    \end{align*}
    where $C>0$ is a constant. There exists $d_2>0$ such that if $\ep 2^{j}\leq d_2$, we have
    \begin{align*}
          \|\mathcal{F}^{-1}[\tilde{k}(\ep 2^{j+l}|\cdot|) \tilde{\phi}(\cdot)]\|_{L^1} &\lesssim \|(1+|x|^2)^{-2}\mathcal{F}^{-1}[(I-\Delta)^2(\tilde{k}(\ep 2^{j+l}|\cdot|) \tilde{\phi}(\cdot))]\|_{L^1}\\
          &\lesssim \|\mathcal{F}^{-1}[(I-\Delta)^2(\tilde{k}(\ep 2^{j+l}|\cdot|) \tilde{\phi}(\cdot))]\|_{L^\infty}\\
          &\lesssim \|(I-\Delta)^2(\tilde{k}(\ep 2^{j+l}|\cdot|) \tilde{\phi})\|_{L^1} \lesssim 1,
    \end{align*}
    where $|l|\leq 1$.
    We also have
    \begin{align*}
          \|\mathcal{F}^{-1}[|\cdot|^3\tilde{\phi}(\cdot)]\|_{L^1} \lesssim 1.
    \end{align*}
    Therefore, there exists a constant $\kappa_1>0$ such that if $\ep 2^{j}\leq d_2$ then
    \begin{align} \label{2}
        \sum_{l=-1}^{1}\|\mathcal{F}^{-1}[\eta_{\pm} \phi^2(2^{-(j+l)}\cdot)]\|_{L^1}  \lesssim e^{\kappa_1 \ep 2^{3j} t}.
    \end{align}
    By Lemma \ref{multi} and the estimates (\ref{1}) and (\ref{2}), there exist constants $d_2>0$ and $\kappa_2>0$ such that if $\ep 2^{j}\leq d_2$, then
    \begin{align*}
        \|\delj \mathcal{F}^{-1}[\eta_{\pm}e^{\lambda_{\pm,0}t}P_{\pm}\hat{\psi}]\|_{L^p} &\lesssim e^{\kappa_1 \ep 2^{3j} t}\|\delj \mathcal{F}^{-1}[e^{\lambda_{\pm,0}t}P_{\pm}\hat{\psi}]\|_{L^p} \\
        &\lesssim e^{\kappa_1 \ep 2^{3j}t} e^{-\kappa_0 2^{2j} t}\| \delj e^{ i \frac{|\nabla|}{\ep}t} \psi\|_{L^{p}}\\
        &\lesssim e^{-c 2^{2j}t}\| \delj e^{i \frac{|\nabla|}{\ep}t} \psi\|_{L^{p}},
    \end{align*}
    where $c>0$ is a constant.
    Thus, we have
    \begin{align} \label{4}
        \|\delj e^{t A_\ep}\psi\|_{L^p} \lesssim e^{-c 2^{2j}t}\|\delj e^{ i\frac{|\nabla|}{\ep}t}\psi\|_{L^{p}}.
    \end{align}
\end{proof}

The following lemma derives the spectrally localized estimate for the semigroup $e^{tA_\ep}$.

\begin{lem} \label{linesti}
    Let $2\leq p<\infty$ and $j\in\mathbb{Z}$. There exist small constants $d_2>0$ and $\ep_0>0$ such that
    \begin{enumerate}[\normalfont{(}i\normalfont{)}]
        \item If $\ep 2^{j}\leq d_2$ and $\ep\leq \ep_0$, then 
        \begin{align}
            \|\delj e^{tA_\ep}\psi\|_{L^p} \lesssim \ep^{1-\frac{2}{p}} \frac{(2^{2j}t)^{\frac{2}{p}}e^{-c 2^{2j}t}}{t} 2^{j\left(2-\frac{8}{p}\right)}\|\delj\psi\|_{L^{p'}},
        \end{align}
        where $c>0$ is a constant.
        \item If $\ep 2^{j}>d_2$ and $\ep\leq \ep_0$, then
        \begin{align}
            \|\delj e^{tA_\ep}\psi\|_{L^p} \lesssim e^{-c\ep^{-2}t} 2^{j3\left(\frac{1}{2}-\frac{1}{p}\right)}\|\delj \psi\|_{L^2},
        \end{align}
        where $c>0$ is a constant.
    \end{enumerate}
\end{lem}

\begin{proof}
    Let $2\leq p<\infty$ and $\ep\leq \ep_0$, where $\ep_0>0$ is a constant appearing in (\ref{upper}). We define the semigroup $e^{tA_{\ep,0}}$ by
    \begin{align*}
        e^{tA_{\ep,0}}\psi = \mathcal{F}^{-1} \left[ e^{t\hat{A}_{\ep,0}(\xi)}\hat{\psi} \right],\ \ \ \psi=(\psi_1,\ldots,\psi_4)^{\mathsf{T}}\in \mathcal{S}'(\mathbb{R}^3)^{4},
    \end{align*}
    where 
    \begin{align*}
        e^{t\hat{A}_{\ep,0}(\xi)}= e^{\lambda_{+,0}t} P_{+}(\xi) + e^{\lambda_{-,0}t} P_{-}(\xi)
        ,\ \ \ \ \lambda_{\pm,0}=-\mu_0 |\xi|^2 \pm i\frac{|\xi|}{\ep}.
    \end{align*}
    By Lemma \ref{linesti0} and Lemma \ref{approeA}, there exists a constant $\kappa_2>0$ such that for any $j\in\mathbb{Z}$ with $\ep 2^{j}\leq d_2$, we have
    \begin{align*}
        \|\delj e^{tA_\ep}\psi\|_{L^{p}} \lesssim \ep^{1-\frac{2}{p}} \frac{(2^{2j}t)^{\frac{2}{p}}e^{-\kappa_2 2^{2j}t}}{t} 2^{2j\left( 1-\frac{4}{p}\right)}\|\delj\psi\|_{L^{p'}}.
    \end{align*}
    Next, we estimate the high frequency part. By using the resolutions $(\ref{spe1})$, $(\ref{eta0})$, 
    for any small $\tilde{d_2}>0$, there exists a constant $\kappa_3=\kappa_3(\tilde{d_2})>0$ such that
    \begin{align} \label{ehigh1}
        |e^{t\hat{A}_{\ep}(\xi)}| \lesssim_{\delta_0} e^{-\kappa_3 \ep^{-2}t}
    \end{align}
    where $\tilde{d_2} \leq \ep|\xi|\leq \mu_0^{-1}+1$, $\xi\in \mathbb{R}^3$.
    By using the identity
    \begin{align*}
        -\mu_0|\xi|^{2} \pm \sqrt{\mu_0^2 |\xi|^{4}-\ep^{-2}|\xi|^2} = -\frac{\ep^{-2}}{\mu_0 \pm \sqrt{\mu_0^2  -\ep^{-2}|\xi|^{-2}}},\ \ \ \xi\in\mathbb{R}^3,
    \end{align*}
    we have, for any $\ep|\xi|\geq \mu_0^{-1}+1$, $\xi\in \mathbb{R}^3$, 
    \begin{align} \label{ehigh2}
        e^{\lambda_{\pm}(\xi)t}&= e^{-\frac{\ep^{-2}}{\mu_0 \pm \sqrt{\mu_0^2  -\ep^{-2}|\xi|^{-2}}}t} \nonumber \\ 
        &\lesssim e^{-\kappa_4 \ep^{-2}t},
    \end{align}
    where $\kappa_4>0$ is a constant. We take $\tilde{d_2}>0$ which satisfies $\mathrm{supp}\,\phi(2^{-j}\cdot)\cap B_{\tilde{d_2}}(0)=\emptyset$ for any $\ep 2^j > d_2$, where $\phi(2^{-j}\cdot)$ is the Fourier multiplier of $\delj$ and $B_{\tilde{d_2}}(0)$ is a ball centered at the origin with radius $\tilde{d_2}$. Then, by (\ref{upper}) and the estimates $(\ref{ehigh1})$, $(\ref{ehigh2})$ we have that if $\ep 2^j>d_2$, then
    \begin{align} \label{jhigh}
        \|\delj e^{tA_\ep}\psi\|_{L^p}\lesssim 2^{j3\left(\frac{1}{2}-\frac{1}{p}\right)}\|\delj e^{tA_\ep}\psi\|_{L^2} \lesssim e^{-c\ep^{-2}t} 2^{j3\left(\frac{1}{2}-\frac{1}{p}\right)}\|\delj \psi\|_{L^2},
    \end{align}
    where $c>0$ is a constant.
\end{proof}
Lemma \ref{linesti} then derives the proof of Proposition \ref{strichartztypeestimate} (ii).
\begin{proof}[Proof of Proposition $\ref{strichartztypeestimate}$ $(ii)$]
    By Lemma \ref{linesti}, we have 
    \begin{align*}
        &\left\|\int_0^t e^{\tau A_\ep}\Psi(t-\tau)d\tau \right\|_{\dot{B}_{p,1}^{s}} \\
        &\hspace{15pt}\lesssim \ep^{1-\frac{2}{p}} \sum_{j}\int_0^t \frac{(2^{2j}\tau)^{\frac{2}{p}}e^{-c 2^{2j}\tau}}{\tau} 2^{j\left(s+2-\frac{8}{p}\right)}\|\delj \Psi(t-\tau)\|_{L^{p'}}d\tau\\
        &\hspace{150pt} +\sum_j \int_0^t e^{-c\frac{\tau}{\ep^2}}2^{j\left(s+\frac{3}{2}-\frac{3}{p}\right)}\|\delj \Psi(t-\tau)\|_{L^{2}}d\tau,
    \end{align*}
    where $c>0$ is a constant.
    Thus, we have
    \begin{align*}
        &\left\|\int_0^t e^{\tau A_\ep}\Psi(t-\tau)d\tau \right\|_{L^r_t(\dot{B}_{p,1}^{s})} \\
        &\hspace{50pt}\lesssim \sup_{j} \int_0^\infty \frac{(2^{2j}\tau)^{\frac{2}{p}}e^{-c 2^{2j}\tau}}{\tau}d\tau \ep^{1-\frac{2}{p}}\|\Psi\|_{L^r_t(\dot{B}^{s+2-\frac{8}{p}}_{p',1})}\\
        &\hspace{200pt}+\int_0^{\infty} e^{-c\frac{\tau}{\ep^2}}d\tau \|\Psi\|_{L^r_t(\dot{B}^{s+\frac{3}{2}-\frac{3}{p}}_{2,1})}\\ 
        &\hspace{50pt} \lesssim \ep^{1-\frac{2}{p}} \|\Psi\|_{L^r_t(\dot{B}^{s+2-\frac{8}{p}}_{p',1}\cap \dot{B}^{s+\frac{3}{2}-\frac{3}{p}}_{2,1})}.
    \end{align*}
\end{proof}

We now show the proof of Proposition $\ref{strichartztypeestimate}$ $(i)$.
\begin{proof}[Proof of Proposition $\ref{strichartztypeestimate}$ $(i)$]
    By Lemma \ref{linesti0}, for any $j\in\mathbb{Z}$, the operator $\{\delo e^{\pm i|\nabla|t}\}_{t\in\mathbb{R}}$ is the bounded family of continuous operators on $L^2$ such that
    \begin{align*}
        \|\delo e^{\pm i|\nabla|t}\delo e^{\mp i|\nabla|t'}\tilde{\psi}\|_{L^\infty} \lesssim \frac{1}{|t-t'|} \|\tilde{\psi}\|_{L^1}
    \end{align*}
    for any $t,t'\in\mathbb{R}$, $t\neq t'$ and $\tilde{\psi}\in L^1$. By applying $T T^*$ argument (see \cite[Theorem 8.18]{MR2768550}),  we have the estimate
    \begin{align*}
        \|\delo e^{\pm i|\nabla|t}\tilde{\psi}\|_{L^{r}_t(L^{p_r})} \lesssim \|\tilde{\psi}\|_{L^{2}}\ \ \ \mathrm{for}\ \ \ \tilde{\psi}\in L^2,
    \end{align*}
    where $1/p_r+1/r=1/2$. Thus, we have the estimate
    \begin{align} \label{stri}
        \|\delj e^{\pm i\frac{|\nabla|}{\ep}t} \psi\|_{L^r_t(L^{p_r})} \lesssim \ep^{\frac{1}{r}}2^{j\frac{2}{r}} \|\delj \psi\|_{L^2}.
    \end{align}
    By Lemma \ref{approeA}, there exists $d_3>0$ such that if $\ep 2^{j}\leq d_3$, then
    \begin{align} \label{jlow2}
        \|\delj e^{t A_\ep}\psi\|_{L^q} \lesssim e^{-\kappa_5 2^{2j}t}\|\delj e^{\pm i\frac{|\nabla|}{\ep}t}\psi\|_{L^{q}}\ \ \ \mathrm{for}\ \ \ 1\leq q \leq\infty,
    \end{align}
    where $\kappa_5>0$ is a constant. By (\ref{jhigh}), if $\ep2^j >d_3$, then we have
    \begin{align} \label{jhigh2}
        \|\delj e^{tA_\ep}\psi\|_{L^q}&\lesssim 2^{j3\left(\frac{1}{2}-\frac{1}{q}\right)}\|\delj e^{tA_\ep}\psi\|_{L^2}\\
        &\lesssim e^{-c\ep^{-2}t} 2^{j3\left(\frac{1}{2}-\frac{1}{q}\right)}\|\delj \psi\|_{L^2}\ \ \ \mathrm{for}\ \ \ 2\leq q <\infty,
    \end{align}
    where $c>0$ is a constant. By (\ref{stri}), (\ref{jlow2}) and (\ref{jhigh2}), if $p_r\leq p$, then we have
    \begin{align} \label{rr}
        \|e^{tA_\ep}\psi\|_{L^{r}_t(\dot{B}_{p,r}^{s})} &\lesssim \|e^{tA_\ep}\psi\|_{L^{r}_t(\dot{B}_{p_r,r}^{s+{3}\left(\frac{1}{2}-\frac{1}{r}-\frac{1}{p}\right)})}\nonumber\\
        &\lesssim \|e^{\pm i\frac{|\nabla|}{\ep}t}\psi\|_{L^{r}_t(\dot{B}_{p_r,r}^{s+{3}\left(\frac{1}{2}-\frac{1}{r}-\frac{1}{p}\right)})} + \|e^{-c\ep^{-2}t}\psi\|_{L^{r}_t(\dot{B}_{2,r}^{s+{3}\left(\frac{1}{2}-\frac{1}{p}\right)})}\nonumber\\
        &\lesssim \ep^{\frac{1}{r}} \|\psi\|_{\Bi{s_1}\cap \Bi{s_2+{3}\left(\frac{1}{2}-\frac{1}{p}\right)}}.
`    \end{align}
    By (\ref{jlow2}) and (\ref{jhigh2}), if $p=2$, then we have
    \begin{align} \label{p2}
        \|e^{tA_\ep}\psi\|_{L^{r}_t(\dot{B}_{2,r}^{s})} &\lesssim \|e^{\mu_0 t\Delta }\psi\|_{L^{r}_t(\dot{B}_{2,r}^{s})} + \|e^{-c\ep^{-2} t }\psi\|_{L^{r}_t(\dot{B}_{2,r}^{s})}\\
        &\lesssim \|\psi\|_{\Bi{s_1}\cap \dot{B}_{2,r}^{s}}
    \end{align}
    since $s_1+r/2<s$.
    By (\ref{rr}) and (\ref{p2}), H\"{o}lder's inequality implies that
    \begin{align} \label{fr}
        \|e^{tA_\ep} \psi \|_{L^{r}_t (\dot{B}^{s}_{p,r})} \lesssim \max\left(\ep^{\frac{1}{r
    }}, \ep^{\frac{1}{2}-\frac{1}{p}}\right)\|\psi\|_{\Bi{s_1}\cap \Bi{s_2+{3}\left(\frac{1}{2}-\frac{1}{p}\right)}}.
    \end{align}
    By interpolating the estimate (\ref{fr}) between $s=\tilde{s_1}$ and $\tilde{s_2}$, where $s_1+2/r<\tilde{s_1}<s<\tilde{s_2}<s_2+3/4$, we obtain the estimate
    \begin{align*}
        \|e^{tA_\ep} \psi \|_{L^{r}_t (\dot{B}^{s}_{p,1})} \lesssim \max\left(\ep^{\frac{1}{r
    }}, \ep^{\frac{1}{2}-\frac{1}{p}}\right)\|\psi\|_{\Bi{s_1}\cap \Bi{s_2+{3}\left(\frac{1}{2}-\frac{1}{p}\right)}}.
    \end{align*}
\end{proof}
{ We now turn to the proof of Theorem \ref{nstthm}. }
\begin{proof}[Proof of Theorem \ref{nstthm}]
Let $c_6= \min(c_4,c_5)$ and let
\begin{align*}
    \|(\sigma_{\ep,0},w_{\ep,0})\|_{\Bi{\frac{1}{2}}\cap \dot{H}^4} + \|F\|_{\Bi{-\frac{3}{2}}\cap \dot{H}^4}\leq c_6,
\end{align*}
where  $c_4>0$ is a constant appearing in Theorem $\ref{time-decay}$ and $c_5>0$ is a constant appearing in Theorem $\ref{iexis}$. Let $(\sigma_\ep,w_\ep)$ be a perturbation obtained in Theorem \ref{time-decay} and let $u$ be a global solution obtained in Theorem \ref{iexis}. 

{ We first show the estimate for the compressible part $(\sigma_\ep,\mathbb{Q}w_{\ep})$.}
Let $V_\ep = (\gamma_2\sigma_\ep, \mathbb{Q}w_\ep)^{\mathsf{T}}$, $N=(\gamma_2 f_\ep, \mathbb{Q}
g_\ep)^{\mathsf{T}}$, where $\gamma_2=P'(\rho_\infty)^{1/2}/\rho_\infty$. Then, the Duhamel principle gives
\begin{align} \label{duhamelrep1}
    V_\ep(t) = e^{tA_{\gamma\ep}} V_{\ep,0} + \int^{t}_{0} e^{\tau A_{\gamma\ep}} N(t-\tau)\,d\tau,
\end{align}
where $V_{\ep,0}=V_\ep(0)$ and $\gamma=P'(\rho_\infty)^{1/2}$.
Let $2\leq p<\infty$, $2<r\leq \infty$ and $1/2+2/r<s<3/p$.
By Proposition \ref{strichartztypeestimate} (i), we have
\begin{align*}
    \|e^{tA_{\gamma\ep}}V_{\ep,0}\|_{L^r_t(\dot{B}_{p,1}^{s})}\lesssim \max\left(\ep^{\frac{1}{r}}, \ep^{\frac{1}{2}-\frac{1}{p}}\right) \|V_{\ep,0}\|_{\Bi{\frac{1}{2}}\cap\dot{H}^4}. 
\end{align*}
By Proposition \ref{strichartztypeestimate} (ii), we have 
\begin{align*}
    \left\|\int_0^t e^{\tau A_{\gamma\ep}}N(t-\tau)d\tau \right\|_{L^r_t(\dot{B}_{p,1}^{s})}  \lesssim \ep^{1-\frac{2}{p}} \|N\|_{L^r_t(\dot{B}^{s+2-\frac{8}{p}}_{p',1}\cap \dot{B}^{s+\frac{3}{2}-\frac{3}{p}}_{2,1})}.
\end{align*}
We show the estimate of $N=(\gamma_2 f_\ep, \mathbb{Q}
g_\ep)^{\mathsf{T}}$. Lemma \ref{bi1} and the decay estimate (\ref{wekest}) in Theorem $\ref{time-decay}$ imply that if $2\leq p <4$, we have 
\begin{align*}
    \|f_\ep\|_{L^r_t(\dot{B}^{s+2-\frac{8}{p}}_{p',1})} &\lesssim \|\dive(v^{*}_\ep+w_\ep)\sigma_\ep+\nabla\sigma^*_\ep \cdot w_\ep\|_{L^r_t(\dot{B}^{s+\left(2-\frac{5}{p}\right)-\frac{3}{p}}_{p',1})}\\
    &\hspace{50pt} +\|(v^{*}_\ep+w_\ep)\cdot \nabla\sigma_\ep+\sigma^*_\ep\dive w_\ep\|_{L^r_t(\dot{B}^{s+\left(2-\frac{5}{p}\right)-\frac{3}{p}}_{p',1})}\\
    &\lesssim_p \|(\sigma^*_\ep,v^*_\ep,w_\ep)\|_{L^{\infty}_t(\Bi{\frac{1}{2}}\cap\dot{H}^3)} \|(\sigma_\ep,v_\ep)\|_{L^r_t(\Bi{s}\cap \dot{H}^3)} \lesssim c_6^2,
\end{align*}
since $1/2+2/r<s<3/p$. If $10<p<\infty$, we have
\begin{align*}
    \|f_\ep\|_{L^r_t(\dot{B}^{s+2-\frac{8}{p}}_{p',1})} &\lesssim \sum_{l=0}^2 \|\nabla^{l}\dive(v^{*}_\ep+w_\ep)\nabla^{2-l}\sigma_\ep\|_{L^r_t(\dot{B}^{s-\frac{5}{p}-\frac{3}{p}}_{p',1})}\\
    &\hspace{50pt}+\sum_{l=0}^2 \| \nabla^l \nabla\sigma^*_\ep \cdot \nabla^{2-l}w_\ep\|_{L^r_t(\dot{B}^{s-\frac{5}{p}-\frac{3}{p}}_{p',1})}\\
    &\hspace{50pt} +\sum_{l=0}^2 \|\nabla^{l}(v^{*}_\ep+w_\ep)\cdot \nabla^{2-l}\nabla\sigma_\ep\|_{L^r_t(\dot{B}^{s-\frac{5}{p}-\frac{3}{p}}_{p',1})}\\
    &\hspace{50pt} +\sum_{l=0}^2 \|\nabla^l \sigma^*_\ep\nabla^{2-l}\dive w_\ep\|_{L^r_t(\dot{B}^{s-\frac{5}{p}-\frac{3}{p}}_{p',1})}\\
    &\lesssim_p \|(\sigma^*_\ep,v^*_\ep,w_\ep)\|_{L^{\infty}_t(\Bi{\frac{1}{2}}\cap\dot{H}^3)} \|(\sigma_\ep,v_\ep)\|_{L^r_t(\Bi{s}\cap \dot{H}^3)} \lesssim c_6^2.
\end{align*}
By H\"{o}lder's inequality, for any $2\leq p <\infty$, we have
\begin{align*}
    \|f_\ep\|_{L^r_t(\dot{B}^{s+2-\frac{8}{p}}_{p',1})} \lesssim_p c_6^{2}.
\end{align*}
By the bilinear estimates in Lemma \ref{bi1}, we have
\begin{align*}
    \|f_\ep\|_{L^r_t(\dot{B}^{s+\frac{3}{2}-\frac{3}{p}}_{2,1})} &\lesssim \|\dive(v^{*}_\ep+w_\ep)\sigma_\ep+\nabla\sigma^*_\ep \cdot w_\ep\|_{L^r_t(\dot{B}^{s+\frac{3}{2}-\frac{3}{p}}_{2,1})}\\
    &\hspace{50pt} +\|(v^{*}_\ep+w_\ep)\cdot \nabla\sigma_\ep+\sigma^*_\ep\dive w_\ep\|_{L^r_t (\dot{B}^{s+\frac{3}{2}-\frac{3}{p}}_{2,1})}\\
    &\lesssim \|(\sigma^*_\ep,v^*_\ep,w_\ep)\|_{L^\infty_t(\Bi{\frac{1}{2}}\cap\dot{H}^3)} \|(\sigma_\ep,v_\ep)\|_{L^r_t(\Bi{s}\cap \dot{H}^3)} \lesssim c_6^2 \\
\end{align*}
By Lemma \ref{bi1}, Lemma \ref{bi2}, Lemma \ref{compos} and the decay estimate ($\ref{wekest}$), we also have
\begin{align} \label{4reason}
    \|g_\ep\|_{L^r_t(\dot{B}^{s+2-\frac{8}{p}}_{p',1}\cap \dot{B}^{s+\frac{3}{2}-\frac{3}{p}}_{2,1})} &\lesssim \|(\sigma^*_\ep, v^*_\ep, \sigma_\ep,w_\ep )\|_{L^\infty_t(\Bi{\frac{1}{2}}\cap\dot{H}^4)} \|(\sigma_\ep, v_\ep)\|_{L^r_t(\Bi{s}\cap \dot{H}^4)} \nonumber\\ 
    &\lesssim c_6^2.
\end{align}
Thus, we obtain
\begin{align} \label{infest}
    \|V_\ep\|_{L^r_t(\dot{B}^{s}_{p,1})} \lesssim \max\left(\ep^{\frac{1}{r}}, \ep^{\frac{1}{2}-\frac{1}{p}}\right) \|V_{\ep,0}\|_{\Bi{\frac{1}{2}}\cap \dot{H}^4}+ \ep^{1-\frac{2}{p}} c_6^2.
\end{align}
{  We next show the estimate for the incompressible part $\mathbb{P}w_\ep - \tilde{u}$, where $\tilde{u}=u-u^*$. The perturbation of incompressible flow $\tilde{u}=u-u^*$ satisfies the following equation:
\begin{equation*}
    \left\{ \,
    \begin{aligned}
        &\rho_\infty(\partial_{t}\tilde{u} + \mathbb{P}\dive(\tilde{u} \otimes u) + \mathbb{P}\dive(u^*\otimes \tilde{u})) = \mu \Delta \tilde{u},\\
        &\dive \tilde{u}=0,\\
        &\tilde{u}|_{t=0}=\mathbb{P}v_{0,\ep}-u^*.
    \end{aligned}
    \right.
\end{equation*}}
Then, the incompressible part $w_\ep^1 = \mathbb{P}w_\ep-\tilde{u}$ satisfies 
\begin{align*}
    \partial_t w_\ep^1 - \frac{\mu}{\rho_\infty} \Delta w_\ep^1 = \frac{1}{\rho_\infty}\mathbb{P}h,
\end{align*}
where $h(\sigma_\ep,\sigma_\ep^*, w_\ep^1,\mathbb{Q}w_\ep, v_\ep^*, \tilde{u}_\ep, u^*_\ep)=h_1+h_2+h_3$,
\begin{align*}
    &h_1=-w_\ep^1 \cdot \nabla v_\ep -\tilde{u}_\ep\cdot\nabla w_\ep^1 - u^*_\ep \cdot \nabla w_\ep^1 \\
    &h_2=-\mathbb{Q}w_\ep \cdot \nabla v_\ep -\tilde{u}_\ep  \cdot \nabla \mathbb{Q}w_\ep - u^*_\ep  \cdot \nabla \mathbb{Q}w_\ep -\tilde{u}_\ep  \cdot \nabla (v_\ep^*-u^*_\ep) - (v_\ep^*-u^*_\ep) \cdot \nabla w_\ep ,\\
    &h_3= (\Psi(\ep\sigma^*_\ep + \ep\sigma_\ep)-\Psi(\ep\sigma^*_\ep))\mathcal{A}_0 (v^*_\ep + w_\ep), \ \ h_4 = (\Psi(\ep\sigma^*_\ep)-\Psi(0)) \mathcal{A}_0 w_\ep,\\
    &\Psi(\zeta)=\frac{1}{\zeta+\rho_\infty}.
\end{align*}

Since $w_\ep^1|_{t=0}=0$, the Duhamel principle gives
\begin{align*}
    w_\ep^1(t) = \int_0^t e^{\frac{\mu}{\rho_\infty} \tau \Delta} \mathbb{P}h(t-\tau) d\tau.
\end{align*}
By Lemma \ref{duah}, we have
\begin{align*}
    \|w_\ep^1\|_{L^r_t(\dot{B}^s_{p,r})} =\left\|\int_0^t e^{\frac{\mu}{\rho_\infty} \tau \Delta} \mathbb{P}h(t-\tau) d\tau\right\|_{L^r_t(\dot{B}^s_{p,r})}
    \lesssim \| h\|_{L^r_t(\dot{B}^{s-2}_{p,r})}.
\end{align*}
By Lemma \ref{bi1} and Lemma \ref{compos}, we have
\begin{align*}
    \|h_1\|_{L^r_t(\dot{B}^{s-2}_{p,r})} \lesssim \|(v_\ep, \tilde{u}_\ep, u^*_\ep)\|_{L^\infty_t (\dot{B}^{\frac{1}{2}}_{2,\infty})} \|w_\ep^1\|_{L^r_t (\dot{B}^s_{p,r})},
\end{align*}
\begin{align*}
    &\|h_2\|_{L^r_t(\dot{B}^{s-2}_{p,r})} \lesssim \|(v_\ep, \tilde{u}_\ep, u^*_\ep)\|_{L^\infty_t (\dot{B}^{\frac{1}{2}}_{2,\infty})} \|\mathbb{Q}w_\ep\|_{L^r_t (\dot{B}^s_{p,r})} \\
    &\hspace{150pt}+ \|(\mathbb{Q}v^*_\ep, \mathbb{P}v^*_\ep-u^*_\ep)\|_{L^\infty_t (\dot{B}^{\frac{1}{2}}_{2,\infty})} \|(\tilde{u},w_\ep)\|_{L^r_t (\dot{B}^s_{p,r})},
\end{align*}
\begin{align*}
    \|h_3\|_{L^r_t(\dot{B}^{s-2}_{p,r})} &\lesssim \|\Psi(\ep \sigma^*_\ep + \ep\sigma_\ep)-\Psi(\ep\sigma^*_\ep)\|_{L^r_t (\dot{B}^s_{p,r})} \|(v_\ep^*, w_\ep)\|_{L^\infty_t(\dot{B}^{\frac{5}{2}}_{2,\infty})}\\
    &\lesssim \ep \|\sigma_\ep\|_{L^r_t (\dot{B}^s_{p,r})}\|(v_\ep^*, w_\ep)\|_{L^\infty_t(\dot{B}^{\frac{5}{2}}_{2,\infty})},
\end{align*}
\begin{align*}
    \|h_4\|_{L^r_t(\dot{B}^{s-2}_{p,r})} &\lesssim \|\Psi(\ep \sigma^*_\ep)-\Psi(0)\|_{L^\infty_t (\dot{B}^{\frac{1}{2}}_{2,\infty})} \|w_\ep\|_{L^r_t(\dot{B}^{s+2}_{p,r})}\\
    &\lesssim \ep \|\sigma_\ep^*\|_{L^\infty_t (\dot{B}^{\frac{1}{2}}_{2,\infty})}\| w_\ep\|_{L^r_t(\dot{B}^{s+2}_{p,r})}.
\end{align*}
Thus, we obtain
\begin{align} \label{le}
    \|w_\ep^1\|_{L^r_t(\dot{B}^{s}_{p,r})} \lesssim \max\left(\ep^{\frac{1}{r}},\ep^{\frac{1}{2}-\frac{1}{p}}\right).
\end{align}
By interpolating the estimate (\ref{le}) between $s=s_1$ and $s_2$, where $1/2+2/r<s_1<s<s_2<3/p$, we obtain the estimate
\begin{align*} 
    \|w_\ep^1\|_{L^r_t(\dot{B}^{s}_{p,1})} \lesssim \max\left(\ep^{\frac{1}{r}},\ep^{\frac{1}{2}-\frac{1}{p}}\right).
\end{align*}
\end{proof}

\section*{Acknowledgement}
    I would like to thank Professor Yoshiyuki Kagei for his helpful discussions and valuable comments. This work was supported by JSPS KAKENHI Grant Number JP23KJ0942.

\bibliographystyle{plain}
\bibliography{refs}
\end{document}